\documentclass[reqno,11pt]{amsart}
\usepackage{amsthm,amsfonts,amssymb,euscript,mathrsfs,graphics,color,amsmath,amssymb,latexsym,marginnote}
\usepackage{soul}
\usepackage{graphicx}

\usepackage{amsthm,amsfonts,amssymb,euscript,mathrsfs,graphics,color,amsmath,amssymb,latexsym,marginnote}

\usepackage{hyperref}

\usepackage[margin=1.0in]{geometry}

\numberwithin{equation}{section}

\def\C{{\mathbb{C}}}

\def\varep{\varepsilon}

\def\sgn{\mathrm{sgn}}

\def\R{{\mathbb R}}

\def\R{{\bf R}}

\def\Z{{\mathbb Z}}

\def\R{\mathbb{R}}

\def\T{{\mathbb T}}

\definecolor{bluegreen}{rgb}{0.0, 0.3, 0.9}

\newtheorem{theorem}{Theorem}[section]
\newtheorem{lemma}[theorem]{Lemma}

\newtheorem{remark}[theorem]{Remark}

\title{Critical well-posedness for the 2D Peskin problem with general tension}
\author{Eduardo Garc\'ia-Ju\'arez and Susanna V. Haziot}

\address{Departamento de An\'alisis Matem\'atico, Universidad de Sevilla, C/Tarfia s/n, Campus Reina Mercedes, 41012, Sevilla, Spain. \href{mailto:egarcia12@us.es}{egarcia12@us.es}}

\address{Department of Mathematics, Brown University, Kassar House,151 Thayer St. Providence, RI 02912, USA. \href{mailto:susanna_haziot@brown.edu}{susanna\_haziot@brown.edu}}

\begin{document}
	\maketitle
\begin{abstract}
    In this paper, we study the two dimensional Peskin problem with general elasticity law. Specifically, we prove global regularity for small perturbations, in suitable critical spaces, of the circle solution, possibly containing corners. For such initial data we prove asymptotic stability in the sense that as $t\to\infty$, the solution converges to a translated and rotated disk.
\end{abstract}
 
	\setcounter{tocdepth}{1}
	\tableofcontents
	\section{Introduction}
	The \textit{immersed boundary method} \cite{Peskin02} was developed by Charles Peskin to study the circulation of blood through the heart valves \cite{Peskin77}. One particular motivation was to understand the interactions between the flow of the fluid and boundary of the heart ventricles with the idea of furthering our comprehension of the mechanisms of valve disease. Since it takes into account the boundary of the heart and the flow of the blood as a single entity, the problem is categorized as an \textit{immersed boundary problem}. 
	
	Due to the extreme complexity of the human heart and its mechanism, in order to launch rigorous mathematical studies on immersed boundary problems, a simpler model was derived by \cite{MoriRodenbergSpirn19} and \cite{LinTong19}, now known as the \textit{Peskin problem}. In two dimensions, the problem consists of two incompressible Stokes fluids separated by a thin, closed, elastic boundary whose shape evolves in time, see Figure~\ref{fig:peskin}. The elasticity of the boundary is captured by a \textit{tension} term, which plays a significant role in the nature of the problem: when the tension is taken to be linear, the problem can be perceived as being semi-linear, whereas when it is nonlinear, the problem becomes fully quasi-linear. Since the problem is scaling invariant in given function spaces, it is said to be \textit{critical} in those given spaces.  
	
	In this paper, we provide a very robust description of the two dimensional problem when the viscosity in the two fluids is the same and allowing for any general, nonlinear tension. Ours is the first global well-posedness result for this problem in critical spaces for this setting, as well as the first for which the initial data allows for corners. We show that, if within a given norm, the initial data is a perturbation of the disk and can contain infinitely many small corners, then there exists a global in time solution for which the corners desingularize and, as time goes to infinity, converges exponentially to a translated circle.  
	
	\subsection{Presentation of the problem}
	We denote by $\Omega_{1,2}$ the two fluid regions in the $(x,y)$-plane, separated by the closed, elastic boundary $\Gamma$. We define by $u_i$ the velocity, $p_i$ the pressure and $\mu_i>0$ the viscosity in the region $\Omega_i$ for $i\in\{1,2\}$. Furthermore, we denote the full velocity field by $\textbf{u}:=(u_1,u_2)$. Since the model considers highly viscous, incompressible fluids, they are modeled by Stokes' flow. 
	\begin{figure}
		\centering
		\includegraphics[scale=0.7]{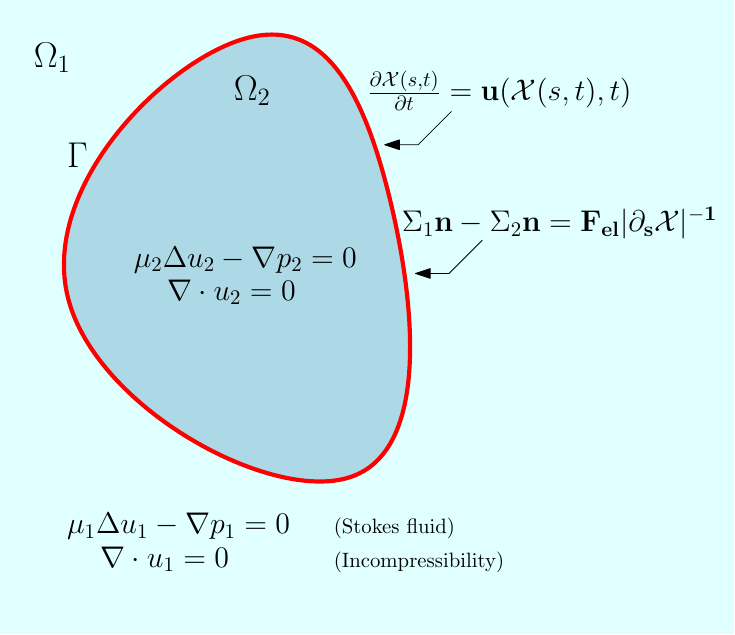}
		\caption{The two-dimensional Peskin problem}
		\label{fig:peskin}
	\end{figure}
	We get
	\begin{equation}\label{present1}
		\begin{aligned}
		\mu_i\Delta u_i-\nabla p_i&=0&\qquad&\text{ in }\Omega_i,\\
		\nabla\cdot u_i&=0&\qquad&\text{ in }\Omega_i,
		\end{aligned}
	\end{equation} 
	for $i\in\{1,2\}$. We must also take into account compatibility conditions on the free boundary $\Gamma$. We begin by parameterizing $\Gamma$ by the Lagrangian coordinate $s\in\mathbb{R}\setminus\{2\pi\mathbb{Z}\}$ and let $\mathcal{X}(s,t)$ denote the coordinate position of $\Gamma$ at time $t$. Denoting by 
	\begin{equation}\label{present2}
		\textbf{n}=-\frac{\partial_s\mathcal{X}}{|\partial_s\mathcal{X}|}
	\end{equation}
	the outward pointing unit normal vector on $\Gamma$, the equations in \eqref{present1} are coupled with the boundary conditions
	\begin{equation}\label{present3}
		\begin{aligned}
		u_1&=u_2&\qquad&\text{ on }\Gamma,\\
		\partial_t\mathcal{X}(s,t)&=\textbf{u}(\mathcal{X}(s,t),t)&\qquad&\text{on }\Gamma,
		\end{aligned}
	\end{equation}
	ensuring that the velocity field is continuous across the boundary and that $\Gamma$ moves with the fluids, and with the stress balance condition
	\begin{equation}\label{present4}
		\Sigma_1\textbf{n}-\Sigma_2\textbf{n}=F_\textup{el}|\partial_s\mathcal{X}|^{-1}\qquad\text{on }\Gamma.
	\end{equation}
	Here, $\Sigma:=(\Sigma_1,\Sigma_2)$ denotes the fluid stress, where, setting $I$ to be the identity matrix, each component is defined by
	\begin{equation}\label{present5}
		\Sigma_i=\mu_i\big(\nabla\textbf{u}+(\nabla\textbf{u})^T-p_iI \big)\qquad\text{in }\Omega_i
	\end{equation}
	and $F_\textup{el}$ denotes the elastic force exerted by $\Gamma$,
 \begin{equation*}
     \begin{aligned}
         F_{el}(s,t):=\partial_s\Big(\mathcal{T}(|\partial_s\mathcal{X}|)\frac{\partial_s\mathcal{X}(s)}{|\partial_s\mathcal{X}(s)|}\Big).
     \end{aligned}
 \end{equation*}
 The tension $\mathcal{T}$ satisfies the structural conditions
 \begin{equation}\label{Tau_cond}
     \mathcal{T}(r)>0,\quad \mathcal{T}'(r)>0.
 \end{equation}
 Physically, \eqref{present4} captures the fact that the force exerted by the viscosity and pressures in the fluids must be balanced by the force resulting from the elasticity of the boundary. To complete the formulation of the problem, we prescribe that both the velocity and the pressure tend to zero in the far-field. Finally, we will restrict ourselves to the setting in which the two viscosities are equal, that is
	\begin{equation}\label{present6}
		\mu_1=\mu_2.
	\end{equation}
	For the rest of the paper, we will take treat $\mathcal{X}=\mathcal{X}(s,t)$ as a complex-valued function in the sense that $\mathcal{X}=\mathcal{X}_1+i\mathcal{X}_2$ for $\mathcal{X}_1,\mathcal{X}_2\in\R$. The problem \eqref{present1}-\eqref{present6} can hence be reformulated as the evolution equation
	\begin{equation}\label{present7}
	\begin{split}
	\partial_t\mathcal{X}(s,t)&=\mathcal{N}(s,t),\\\mathcal{X}_0(s)&=\mathcal{X}(s,0),
	\end{split}
	\end{equation}
	where we define
	\begin{equation}\label{present8}
		\mathcal{N}(s,t):=\frac{1}{4\pi}\mathrm{p.v.} \int_{\T}\operatorname{Re}\bigg[\frac{\mathcal{X}'(r,t)^2}{(\mathcal{X}(r,t)-\mathcal{X}(s,t))^2}\bigg](\mathcal{X}(r,t)-\mathcal{X}(s,t))T(|\mathcal{X}'(r,t)|)\,dr,
	\end{equation}
	for $\mathcal{X}:\T\to\C$ taken to be a $C^2$ function and $T:\R\to\R$, capturing the tension along the elastic boundary, 
 to be a smooth function away from the origin,
  \begin{equation}\label{T_def}
     \begin{aligned}
         T(r):=\frac{\mathcal{T}(r)}{r}.
     \end{aligned}
 \end{equation}
	
	The two-dimensional Peskin problem is scaling invariant in the sense that if $\mathcal{X}(s,t)$ is a solution to \eqref{present7}, then so is $\lambda^{-1}\mathcal{X}(\lambda s,\lambda t)$ for any $\lambda>0$. The spaces $\dot{H}^{3/2}$ and $\dot{W}^{1,\infty}$ are such \textit{critical spaces} in which this scaling invariance holds. Hence, in order to construct solutions whose initial data has corners, we are forced to work in the critical space $\dot{W}^{1,\infty}$.  
	
	\subsection{Statement of the main results}  The main result of the paper will be to prove global well-posedness of the initial-value problem \eqref{present7} in function spaces for which the initial data allows for small corners. To this end, we introduce the following spaces: for the initial data we assume that $\|f\|_{S}\leq\varepsilon_0$, where
	\begin{equation}\label{norm1}
	\|f\|_{S}:=\|f'\|_{L^\infty}+\sup_{n\in\Z_+}2^{3n/2}\|P_nf\|_{L^2},
	\end{equation}
	with $P_n$ the usual Littlewood-Paley projections (defined by Fourier multipliers $\xi\to\varphi_n(\xi)$, see for example \cite{BahouriCheminDanchin11}). Notice that, in terms of Besov spaces, $S=\dot{W}^{1,\infty}\cap \dot{B}^{\frac32}_{2,\infty}$, and this space allows for multiple corners. Then we define the space $Z_1$ as (essentially) the space of linear evolutions with initial data in $S$, induced by the norm
	\begin{equation}\label{norm2}
	\|F\|_{Z_1}:=\sup_{t\in[0,\infty)}(1+t)^{2/3}\|F'(t)\|_{L^\infty_x}+\sup_{t\in[0,\infty)}\sup_{n\in\Z_+}(1+2^nt)^{2/3}2^{3n/2}\|P_nF(t)\|_{L^2}.
	\end{equation}
	Finally we define the space $Z_2$ where we plan on measuring the nonlinearity by the norm
	\begin{equation}\label{norm3}
	\|F\|_{Z_2}:=\sup_{t\in[0,\infty)}\sup_{n\in\Z_+}(2^nt)^{-1/3}(1+2^nt)2^{3n/2}\|P_nF(t)\|_{L^2}.
	\end{equation}
	Notice that $Z_2$ is a stronger norm than $Z_1$, i.e. $\|F\|_{Z_1}\lesssim \|F\|_{Z_2}$ for any function $F$. 
	
	We are now in a position to state the main theorem.
	\begin{theorem}
		Let $\varep\ll1$ and let $\mathcal{X}_0$ be such that
		\begin{equation*}
			\|\mathcal{X}_0(s)-e^{is}\|_{S}\leq\varep,
		\end{equation*}
		where the space $S$ is defined as in \eqref{norm1}. Then there exists a unique solution $\mathcal{X}:[0,\infty)\times\T\to\C$ to the initial value problem \eqref{present7} of the form $$\mathcal{X}(s,t)=a_0(t)+a_1(t)e^{is}+\mathcal{Y}(s,t),$$ with $\|\mathcal{Y}\|_{Z_1}\leq\varep$, which converges exponentially to the translated circle $\widetilde{a_0}+\widetilde{a_1}e^{is}$. Here,
		\begin{equation*}
			\widetilde{a_0}:=\lim_{t\to\infty}a_0(t)\qquad \text{and}\qquad \widetilde{a_1}:=\lim_{t\to\infty}a_1(t).
		\end{equation*}  
	\end{theorem}
	\begin{figure}
		\centering
		\includegraphics[scale=0.7]{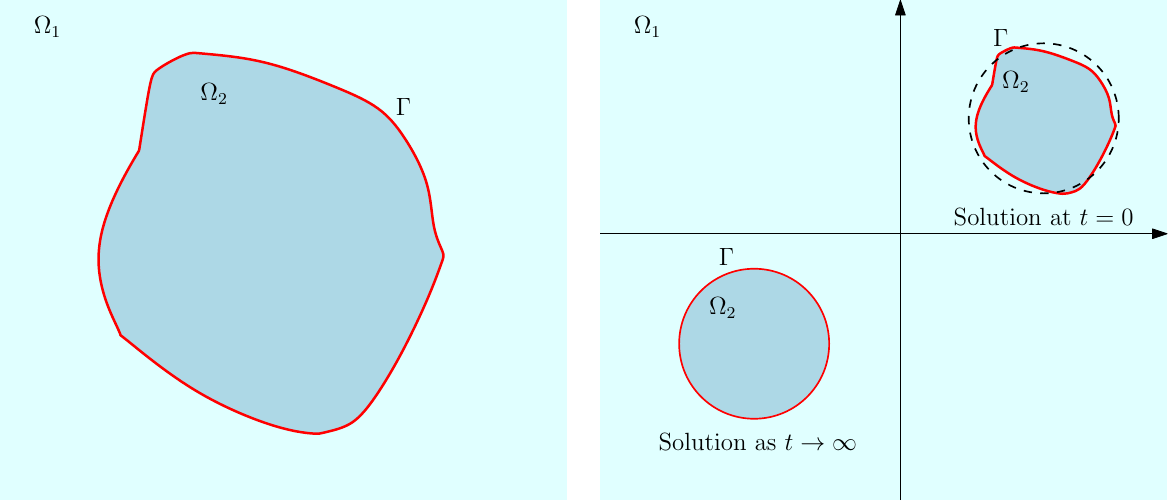}
		\caption{The initial data with small corners and its evolution as $t\to\infty$}
		\label{fig:initial}
	\end{figure}
	\begin{remark}
		Informal stated, if the initial data consisting of small corners is a small perturbation of the disk, then there exists a global solution for which the corners desingularize and which converges to a translated disk as time goes to infinity, see Figure~\ref{fig:initial}. 
	\end{remark}
	\begin{remark}
		Although our problem is quasi-linear, we use a fixed point argument with Duhamel's formula in order to construct the solution. This is possible due to the a certain parabolic characteristic of the problem: the linear evolution does not contain a clear-cut multiplier (such as $-|\nabla|$) to capture the dissiptation, but instead, since the various Fourier modes interact, we have a matrix which can be shown acts analytically in the same way as $-|\nabla|$. Moreover, the first two modes, $a_0$ and $a_1$ are not dissipative, and hence must be removed before performing the fixed point argument. This is what captures the \textit{orbital stability}.
	\end{remark}
	\begin{remark}
		The $\mathcal{Y}$ portion of the solution consists of a linear flow which lies in the $Z_1$ space and a nonlinear integral term due to Duhamel's formula which lies in the $Z_2$ space. Since the $Z_1$ space is a larger space than the $Z_2$ space, we say that $\mathcal{Y}$ lies in $Z_1$.
	\end{remark}
	We are also able to prove a similar result in the Wiener algebra space $W$, defined by the norm 
	\begin{equation}\label{W_norm}
	\|F\|_W:=\sup_{t\in[0,\infty)}\{\sum_{k\in\Z}|\widehat{F}(k)||k|\}+\sup_{t\in[0,\infty),\,m\in\Z_+}\{\sum_{|k|\in [2^{m-1},2^{m+1}]}|\widehat{F}(k)||k|(1+|k|t)^{2/3}\}
	\end{equation}
	with corresponding $N$ norm for the Fourier coefficients induced by
	\begin{equation}\label{N_2_norm}
	\|\{a_k\}_{k\in\Z}\|_N:=\sup_{t\in[0,\infty)}\{\sum_{k\in\Z}|a_k|\}+\sup_{t\in[0,\infty),\,m\in\Z_+}\{\sum_{|k|\in[2^{m-1},2^{m+1}]}|a_k|(1+|k| t)^{2/3} \}.
	\end{equation}
	We get the following theorem.
	\begin{theorem}
		Let $\varep\ll1$ and let $\mathcal{X}_0$ be such that
		\begin{equation*}
		\|\mathcal{X}_0(s)-e^{is}\|_{W}\leq\varep,
		\end{equation*}
		where the space $W$ is defined as in \eqref{W_norm}. Then there exists a unique solution $\mathcal{X}:[0,\infty)\times\T\to\C$ to the initial value problem \eqref{present7} of the form $$\mathcal{X}(s,t)=a_0(t)+a_1(t)e^{is}+\mathcal{Y}(s,t),$$ with $\|\mathcal{Y}\|_{W}\leq\varep$, which converges exponentially to the translated circle $\widetilde{a_0}+\widetilde{a_1}e^{is}$. Here,
		\begin{equation*}
		\widetilde{a_0}:=\lim_{t\to\infty}a_0(t)\qquad \text{and}\qquad \widetilde{a_1}:=\lim_{t\to\infty}a_1(t).
		\end{equation*}  
	\end{theorem}

	\subsection{Historical considerations}

The breakthrough results of Lin and Tong \cite{LinTong19} and Mori, Rodenberg and Spirn \cite{MoriRodenbergSpirn19} initiated the rigorous study of the Peskin problem. In both works, the authors consider the case of linear elasticity, which results in a \textit{semilinear} system of equations for the evolution of the elastic filament. Lin and Tong used energy estimates in combination with the Schauder fixed point theorem to show local well-posedness in $H^{5/2}$, as well as exponential convergence to the equilibrium for initial data sufficiently close to the steady states. Via a semigroup approach, Mori, Rodenberg and Spirn were able to obtain local well-posedness in barely subcritical spaces on the little H\"older scale, $h^{1,\gamma}$, $\gamma\in(0,1)$, (i.e., the completion of smooth functions in classical H\"older spaces), instant smoothing of the free boundary and found the exact exponential rate of convergence towards the equilibrium, together with a criterion for global-in-time existence. 

After these results, the main research effort has been placed on lowering the regularity required for the initial data. In the semilinear setting,  García-Juárez, Mori and Strain \cite{GJMoriStrain23} showed global well-posedness for initial data close to equilibrium in a critical setting, namely with the tangent vector pertaining to the Wiener algebra (thus bounded and continuous). Then, Gancedo, Granero and Scrobogna \cite{GancedoGraneroScrobogna21}, introduced a toy model to capture the behavior of the normal component, for which global weak solutions were constructed in the Lipschitz class for sufficiently small initial data. Still in the semilinear setting, Chen and Nguyen 
\cite{ChenNguyen23} reached global well-posedness in the Besov space $\dot{B}^1_{\infty, \infty}$, which is strictly bigger than the Lipschitz class. The authors also proved local well-posedness for arbitrary initial curves, but required a smoother space: the closure of $C^2$ in $\dot{B}^1_{\infty, \infty}$. The proof consists of a fixed point scheme in customized spaces which roughly behave like the Chemin-Lerner space $\tilde{L}_T^{1/\alpha}(\dot{B}_{\infty,\infty}^{1,\alpha})$.

In a different direction, later, Tong \cite{Tong22} considered another simplified model: an infinitely long string that moves only due to tangential stretching. Interestingly, global weak solutions are constructed in the energy class, i.e., $\mathcal{X}\in H^1$, without smallness conditions. Building upon this work, Tong and Wei \cite{TongWei23} have very recently proved a global well-posedness result for the Peskin problem in $C^{1,\gamma}$, where only smallness in the normal deviation from the steady states is needed. Among other related works, Tong 
\cite{Tong21} studied, for computational purposes, a regularized Peskin problem and its convergence to the actual one, Li \cite{Li21} considered the effect of bending and stretching, and  García-Juárez, Kuo, Mori and Strain \cite{GJKuoMoriStrain23} initiated the study of the 3D Peskin problem.

Results for the setting with nonlinear elasticity are much more scarce. In the subcritical $h^{1,\gamma}$ class, Rodenberg \cite{Rodenberg18} proved local well-posedness based on abstract semigroup theory. More recently, the remarkable work by Cameron and Strain \cite{CameronStrain23} showed local well-posedness and smoothing of the solution in the critical Besov space $\dot{B}_{2,1}^{3/2}$. However, as for the Wiener algebra, functions in $\dot{B}_{2,1}^{1/2}$ are continuous and hence this class does not allow for interfaces with corners. None of the previous works studies the global well-posedness and stability near steady states.
When the tension law is linear, the semilinear system has a constant-coefficient parabolic leading term, which is basically the square-root of the Laplacian. Together with the extra regularity enjoyed by the lower-order remaining nonlinear terms, this fact makes it possible to obtain well-posedness even for initial interfaces with unbounded tangent vectors \cite{ChenNguyen23}.
This is no longer true for a nonlinear tension law: the system becomes quasilinear and the dissipative term can become degenerate if the tangent vector is not uniformly bounded.

As a matter of fact, this lack of results in the nonlinear, critical setting can be compared with the state-of-the-art in the heavily-studied Muskat problem. The Muskat problem is a fluid-interface problem, closely related in its mathematical structure to the Peskin problem, where the evolution equation for the free boundary is also quasilinear. The latest results on the Cauchy problem for the Muskat equation by Alazard and Nguyen \cite{AlazardNguyen21}, \cite{AlazardNguyen22}, whose techniques inspired the Peskin work by Cameron and Strain \cite{CameronStrain23}, provided critical local and global well-posedness in $W^{1,\infty}\cap H^{3/2}$ and solely $H^{3/2}$, respectively. Very recently, Gómez-Serrano, Pausader and the two authors \cite{GJGSHP23} showed global well-posedness and smoothing in critical spaces which include interfaces with corners. 

We notice that the equilibrium state for the Muskat equation in these works corresponds to the zero function (i.e. a flat interface), while in Peskin the steady states form a four-dimensional vector space (the uniformly parametrized circles). The existence of zero eigenvalues makes the analysis more involved, and in fact the closest setting for the Muskat problem would be that of bubble, for which only subcritical results are available \cite{GancedoGarciaJuarezPatelStrain23}.

	\subsection{Outline}
In Section~\ref{sect:linear}, we begin by studying the linear problem. After linearizing the solution around the disk, we study the output Fourier modes. We make two observations: firstly, we have no zeroth or first order modes in the output. These two modes don't decay at infinity, but instead, converge to some different value, capturing the motion of the equilibrium stead as $t\to\infty$. As a result, we can only treat the problem as a parabolic problem after we remove these two modes. Moreover, we notice that the remaining modes interact. We then show in Section~\ref{sect:duhamel} that the resulting matrix analytically acts like $|\nabla|$. 

In Section~\ref{sect:wiener}, as a warm-up result, we prove well-posedness of the problem in the Wiener algebra space and subsequently carry out the fixed point argument using an iteration of Duhamel's formula in Section~\ref{sect:duhamel}, treating the problem one mode at a time. The remainder of the paper concerns well-posedness in spaces allowing for initial data with small corners. In Section~\ref{sect:corners}, we reformulate the nonlinearity as a multilinear paraproduct. Since we work on the torus, the study of the kernels requires the use of the Poisson summation formula in order to handle the absence of integration by parts one would otherwise use for the real line in order to get the necessary bounds. Section~\ref{sect:estimates} then acts as a collection of estimates which will be fundamental in bounding the nonlinearity. In fact, this section contains the heart of the proof: by means of these estimates,  we are able to show in Section~\ref{sect:final} that analytically, the nonlinearity can be percieved as a Hilbert transform multiplied by a number of copies of $Y'$, allowing us to get the necessary bounds in order to close the fixed point argument.  	
	  
	\section{Study of the linear problem}\label{sect:linear}
	\subsection{Preliminaries}
	
	We recall the evolution problem we consider is of the following form:
	\begin{equation}\label{intro1}
	\begin{split}
	&\partial_t\mathcal{X}=\mathcal{N},\\
	&\mathcal{N}(s):=\frac{1}{4\pi}\mathrm{p.v.} \int_{\T}\operatorname{Re}\Big[\frac{\mathcal{X}'(r)^2}{(\mathcal{X}(r)-\mathcal{X}(s))^2}\Big](\mathcal{X}(r)-\mathcal{X}(s))T(|\mathcal{X}'(r)|)\,dr,
	\end{split}
	\end{equation}
	where $\mathcal{X}:\T\to\C$ is a $C^2$ function and $T:\R\to\R$ is a smooth function. We are looking for solutions $\mathcal{X}$ of the form
	\begin{equation}\label{intro2}
	\mathcal{X}(r)=e^{ir}+X(r),\qquad X:\T\to\C.
	\end{equation}
	
	Then we calculate
	\begin{equation}\label{intro3}
	\mathcal{X}'(r)=ie^{ir}+X'(r)=ie^{ir}(1-ie^{-ir}X'(r)),
	\end{equation}
	\begin{equation}\label{intro4}
	\mathcal{X}(r)-\mathcal{X}(s)=e^{i(r+s)/2}[e^{-i(s-r)/2}-e^{i(s-r)/2}+e^{-i(r+s)/2}(X(r)-X(s))],
	\end{equation}
	and
	\begin{equation}\label{intro5}
	\frac{\mathcal{X}'(r)^2}{(\mathcal{X}(r)-\mathcal{X}(s))^2}=\frac{e^{-i(s-r)}[1-ie^{-ir}X'(r)]^2}{[2\sin((s-r)/2)+ie^{-i(r+s)/2}(X(r)-X(s))]^2}.
	\end{equation}
	It will be convenient to let
	\begin{equation}\label{intro6}
	\widetilde{X}(s,r):=\frac{e^{-i(s+r)/2}(X(r)-X(s))}{2\sin((s-r)/2)},
	\end{equation}
	and hence, the last two formulas become
	\begin{equation}\label{intro7}
	\begin{split}
	&\mathcal{X}(r)-\mathcal{X}(s)=(-2)ie^{i(r+s)/2}\sin((s-r)/2)(1+i\widetilde{X}(s,r)),\\
	&\frac{\mathcal{X}'(r)^2}{(\mathcal{X}(r)-\mathcal{X}(s))^2}=\frac{e^{-i(s-r)}(1-ie^{-ir}X'(r))^2}{4\sin^2((s-r)/2)(1+i\widetilde{X}(s,r))^2}.
	\end{split}
	\end{equation}
	
	We can thus reformulate the function $\mathcal{N}$ defined in \eqref{intro1} as
	\begin{equation}\label{intro8}
	\begin{split}
	\mathcal{N}(s):=\frac{-i}{4\pi}\mathrm{p.v.} \int_{\T}&\Re\Big[\frac{e^{-i(s-r)}(1-ie^{-ir}X'(r))^2}{(1+i\widetilde{X}(s,r))^2}\Big]\frac{e^{i(s+r)/2}(1+i\widetilde{X}(s,r))}{2\sin((s-r)/2)}T(|1-ie^{-ir}X'(r)|)\,dr.
	\end{split}
	\end{equation}

	The following lemma will be very useful for studying the linear part of the equation.
	\begin{lemma}\label{introLemma}
		For any $k\in\Z$ we have
		\begin{equation}\label{intro11}
		I_k:=\frac{1}{2\pi}\mathrm{p.v.} \int_{\T}\frac{e^{-i\alpha/2}}{2\sin(\alpha/2)}e^{-ik\alpha}\,d\alpha=
		\begin{cases}
		-i/2\qquad&\text{ if }k\geq 0,\\
		i/2\qquad&\text{ if }k\leq -1.
		\end{cases}
		=:\frac{-i}{2}\sgn(k),
		\end{equation}
		and
		\begin{equation}\label{intro12}
		J_k:=\frac{1}{2\pi}\mathrm{p.v.} \int_{\T}\frac{e^{-ik\alpha}-1}{4\sin(\alpha/2)^2}\,d\alpha=-|k|/2.
		\end{equation}
	\end{lemma}
	
	\begin{proof} For $k\in\{0,-1\}$ we have
		\begin{equation*}
		I_0=\frac{1}{2\pi}\mathrm{p.v.}\int_{-\pi}^\pi\frac{e^{-i\alpha/2}}{2\sin(\alpha/2)}\,d\alpha=\frac{1}{2\pi}\mathrm{p.v.}\int_{-\pi}^\pi\frac{\cos(\alpha/2)-i\sin(\alpha/2)}{2\sin(\alpha/2)}\,d\alpha=-i/2,
		\end{equation*}
		and
		\begin{equation*}
		I_{-1}=\frac{1}{2\pi}\mathrm{p.v.}\int_{-\pi}^\pi\frac{e^{-i\alpha/2}}{2\sin(\alpha/2)}e^{i\alpha}\,d\alpha=\frac{1}{2\pi}\mathrm{p.v.}\int_{-\pi}^\pi\frac{\cos(\alpha/2)+i\sin(\alpha/2)}{2\sin(\alpha/2)}\,d\alpha=i/2,
		\end{equation*}
		due to the oddness of the function $\alpha\to\cos(\alpha/2)/\sin(\alpha/2)$. Moreover, if $k\geq 1$ then
		\begin{equation*}
		I_k-I_{k-1}=\frac{1}{2\pi}\int_{\T}\frac{e^{-i\alpha/2}}{2\sin(\alpha/2)}[e^{-ik\alpha}-e^{-i(k-1)\alpha}]\,d\alpha=\frac{1}{2\pi}\int_{\T}\frac{e^{-ik\alpha}[e^{-i\alpha/2}-e^{i\alpha/2}]}{2\sin(\alpha/2)}\,d\alpha=0.
		\end{equation*}
		Finally, if $k\leq -2$ then
		\begin{equation*}
		I_k-I_{k+1}=\frac{1}{2\pi}\int_{\T}\frac{e^{-i\alpha/2}}{2\sin(\alpha/2)}[e^{-ik\alpha}-e^{-i(k+1)\alpha}]\,d\alpha=\frac{1}{2\pi}\int_{\T}\frac{e^{-i(k+1)\alpha}[e^{i\alpha/2}-e^{-i\alpha/2}]}{2\sin(\alpha/2)}\,d\alpha=0.
		\end{equation*}
		The desired conclusions \eqref{intro11} follow.
		
		To prove \eqref{intro12} we notice that $J_0=0$ and $J_{-k}=J_k$ (by parity considerations). For $k\geq 1$ we write 
		\begin{equation*}
		e^{-ik\alpha}-1=(e^{-i\alpha}-1)(1+e^{-i\alpha}+\ldots+e^{-i(k-1)\alpha})=-2i\sin(\alpha/2)e^{-i\alpha/2}\sum_{j=0}^{k-1} e^{-ij\alpha},
		\end{equation*}
		therefore, using \eqref{intro11},
		\begin{equation*}
		J_k=\frac{1}{2\pi}\mathrm{p.v.}\int_{\T}\frac{-ie^{-i\alpha/2}}{2\sin(\alpha/2)}\sum_{j=0}^{k-1} e^{-ij\alpha}\,d\alpha=-k/2. 
		\end{equation*}
		The desired identities \eqref{intro12} follow.
	\end{proof}

	\subsection{The linearized problem} 
	We begin with a further reduction. We write
	\begin{equation}\label{freq1}
	X(r)=a_0+a_1e^{ir}+\sum_{k\neq\{0,1\}}a_ke^{ikr}=a_0+a_1e^{ir}+Y(r).
	\end{equation}
	Then we can easily see that
	\begin{equation}\label{freq2}
	X'(r)=ia_1e^{ir}+Y'(r)
	\end{equation}
	as well as
	\begin{equation}\label{freq3}
	\widetilde{X}(s,r)=\frac{e^{-i(s+r)/2}(Y(r)-Y(s))}{2\sin((s-r)/2)}+\frac{e^{-i(s+r)/2}(a_1e^{ir}-a_1e^{is})}{2\sin((s-r)/2)}=\widetilde{Y}(s,r)-ia_1.
	\end{equation}

	Using \eqref{freq2} and \eqref{freq3}, we can rewrite \eqref{intro8} as
	\begin{equation}\label{freq7}
	\begin{split}
	\mathcal{N}(s)=\frac{-i}{4\pi}\mathrm{p.v.} \int_{\T}&\Re\Big[\frac{e^{-i(s-r)}(1\!-\!\frac{ie^{-ir}Y'(r)}{1+a_1})^2}{(1+\frac{i\widetilde{Y}(s,r)}{1+a_1})^2}\Big]\frac{e^{i(s+r)/2}(1\!+\!a_1\!+\!i\widetilde{Y}(s,r))}{2\sin((s-r)/2)}T(|1\!+\!a_1\!-\!ie^{-ir}Y'(r)|)\,dr.
	\end{split}
	\end{equation}
	Linearizing $T(\cdot)$ in terms of $Y$ and denoting $A=A(a_1)=T(|1+a_1|)$, $B=B(a_1)=T'(|1+a_1|)$, we now extract the $0$-th order and first order components of the function $\mathcal{N}$ to find
	\begin{equation}\label{freq8}
	\begin{split}
	\mathcal{N}_0(s)&:=\frac{-i}{4\pi}\mathrm{p.v.} \int_{\T}\Re\big[e^{-i(s-r)}\big]\frac{e^{i(s+r)/2}(1+a_1)}{2\sin((s-r)/2)}A\,dr\\
	&=\frac{-iA(1+a_1)}{4\pi}\mathrm{p.v.} \int_{\T}\Re\big[e^{i\alpha}\big]\frac{e^{i(2s-\alpha)/2}}{2\sin(\alpha/2)}\,d\alpha\\
	&=0,
	\end{split}
	\end{equation}
	by using Lemma~\ref{introLemma} in the last line.

	For the first order component, we get
	\begin{equation}\label{freq9}
	\begin{split}
	\mathcal{N}_1(s)&:=\frac{-i}{4\pi}\mathrm{p.v.} \int_{\T}\Re\bigg[\frac{e^{-i(s-r)}(-2)i}{1+a_1}[e^{-ir}Y'(r)+\widetilde{Y}(s,r)]\bigg]\frac{e^{i(s+r)/2}(1+a_1)}{2\sin((s-r)/2)}A\,dr\\
	&+\frac{-i}{4\pi}\mathrm{p.v.} \int_{\T}\Re\big[e^{-i(s-r)}\big]\frac{e^{i(s+r)/2}i\widetilde{Y}(s,r)}{2\sin((s-r)/2)}A\,dr\\
	&+\frac{-i}{4\pi}\mathrm{p.v.} \int_{\T}\Re\big[e^{-i(s-r)}\big]\frac{e^{i(s+r)/2}(1+a_1)}{2\sin((s-r)/2)}\frac{B[-ie^{-ir}Y'(r)(1+\overline{a_1})+ie^{ir}\overline{Y'(r)}(1+a_1)]}{2|1+a_1|}\,dr.
	\end{split}
	\end{equation}
	We make the change of variables $r=s-\alpha$ and simplify slightly the expressions to rewrite
	\begin{equation}\label{freq10}
	\begin{split}
	&\mathcal{N}_1=I+II+II,\\
	&I(s):=A\frac{ie^{is}}{2\pi}\mathrm{p.v.} \int_{\T}\Re\bigg[ie^{-is}\frac{Y'(s-\alpha)}{1+a_1}+ie^{-i\alpha}\frac{\widetilde{Y}(s,s-\alpha)}{1+a_1}\bigg]\frac{e^{-i\alpha/2}(1+a_1)}{2\sin(\alpha/2)}\,d\alpha,\\
	&II(s):=A\frac{e^{is}}{4\pi}\mathrm{p.v.} \int_{\T}\Re\big[e^{-i\alpha}\big]\frac{e^{-i\alpha/2}\widetilde{Y}(s,s-\alpha)}{2\sin(\alpha/2)}\,d\alpha,\\
	&III(s):=B\frac{e^{is}(1+a_1)}{8\pi}\mathrm{p.v.} \int_{\T}\Re\big[e^{-i\alpha}\big]\frac{e^{-i\alpha/2}}{2\sin(\alpha/2)}\frac{-2i\Im{\big(e^{-i(s-\alpha)}Y'(s-\alpha)(1+\overline{a_1})\big)}}{|1+a_1|}d\alpha.
	\end{split}
	\end{equation}
	Assuming now that $Y(r)=a_ke^{ikr}$ for some $k\in\Z \setminus\{0,1\}$, we use Lemma~\ref{introLemma} to calculate
	\begin{equation*}
	\begin{split}
	I(s)&=A\frac{ie^{is}}{2\pi}\mathrm{p.v.} \int_{\T}\frac{e^{-i\alpha/2}}{2\sin(\alpha/2)}\frac{1}{2}\Big[-ka_ke^{-ik\alpha}e^{i(k-1)s}-k\overline{a_k}e^{ik\alpha}e^{-i(k-1)s}\frac{1+a_1}{1+\overline{a_1}}\\
	&\qquad\qquad+\frac{ia_ke^{-i\alpha/2}(e^{-ik\alpha}-1)e^{i(k-1)s}}{2\sin(\alpha/2)}-\frac{i\overline{a_k}e^{i\alpha/2}(e^{ik\alpha}-1)e^{-i(k-1)s}}{2\sin(\alpha/2)}\frac{1+a_1}{1+\overline{a_1}}\Big]\,d\alpha
\end{split}
\end{equation*}
 which simplifies to
 	\begin{equation*}
	\begin{split}
	I(s)&=-\frac{1}{2}e^{iks}Aika_kI_k-\frac{1}{2}e^{-i(k-2)s}Aik\overline{a_k}I_{-k}\frac{1+a_1}{1+\overline{a_1}}-\frac{1}{2}e^{iks}Aa_k(J_{k+1}-J_1)\\
 &\quad+\frac{1}{2}e^{-i(k-2)s}A\overline{a_k}J_{-k}\frac{1+a_1}{1+\overline{a_1}}\\
	&=-\frac{1}{4}e^{iks}A|k|a_k+\frac{1}{4}e^{-i(k-2)s}A|k|\overline{a_k}\frac{1+a_1}{1+\overline{a_1}}+\frac{1}{4}e^{iks}A(|k+1|-1)a_k\\
 &\quad-\frac{1}{4}e^{-i(k-2)s}A|k|\overline{a_k}\frac{1+a_1}{1+\overline{a_1}},
	\end{split}
	\end{equation*}
	\begin{equation*}
	\begin{split}
	II(s)&=A\frac{e^{is}}{4\pi}\mathrm{p.v.} \int_{\T}\frac{e^{-i\alpha}+e^{i\alpha}}{2}\frac{e^{-i\alpha/2}}{2\sin(\alpha/2)}\frac{a_ke^{i\alpha/2}(e^{-ik\alpha}-1)e^{i(k-1)s}}{2\sin(\alpha/2)}\,d\alpha\\
	&=-\frac{1}{8}e^{iks}A(|k+1|+|k-1|-2)a_k,
	\end{split}
	\end{equation*}
	and
	\begin{equation*}
	\begin{split}
	III(s)&=B\frac{e^{is}(1+a_1)}{8\pi}\mathrm{p.v.} \int_{\T}\frac{e^{-i\alpha}+e^{i\alpha}}{2}\frac{e^{-i\alpha/2}}{2\sin(\alpha/2)}\\
 &\qquad\times\big[-ika_k\frac{(1+\overline{a_1})}{|1+a_1|}e^{i(k-1)(s-\alpha)}-ik\overline{a_k}\frac{(1+a_1)}{|1+a_1|}e^{-i(k-1)(s-\alpha)}\big]\,d\alpha\\
	&=-ika_k\frac{Be^{iks}|1+a_1|}{8}(I_k+I_{k-2})-ik\overline{a_k}\frac{Be^{-i(k-2)s}(1+a_1)^2}{8|1+a_1|}(I_{-k+2}+I_{-k})\\
	&=-\frac{1}{16}e^{iks}Ba_kk\big(\sgn(k)+\sgn(k-2)\big)|1+a_1|-\frac{1}{16}e^{-i(k-2)s}B\overline{a_k}k(\sgn(-k+2)\\
 &\quad+\sgn(-k))\frac{(1+a_1)^2}{|1+a_1|},
	\end{split}
	\end{equation*}
 
 where we define $\sgn(k)$ as in \eqref{intro11}.
	Therefore
	\begin{equation*}
	\begin{split}
	\mathcal{N}_1(s)=&-\frac{1}{8}A(a_ke^{iks})(2|k|+|k-1|-|k+1|)\\
	&-\frac{1}{8}B(a_ke^{iks})|1+a_1|(|k|-\delta_1(k))+\frac{1}{8}B\frac{(1+a_1)^2}{|1+a_1|}(\overline{a_k}e^{-i(k-2)s})(|k|-\delta_1(k)-2\delta_2(k)).
	\end{split}
	\end{equation*}
	By linearity, in the general case, if
	\begin{equation}\label{freq11}
	Y(t,s)=\sum_{k\neq\{0,1\}}a_k(t)e^{iks}
	\end{equation}
	then
	\begin{equation}\label{freq12}
	\begin{split}
	\mathcal{N}_1(t,s)&=-\sum_{k\neq\{0,1\}}\big\{\frac{1}{8}A(a_k(t)e^{iks})(2|k|+|k-1|-|k+1|)\big\}\\
	&\quad-\sum_{k\neq\{0,1\}}\big\{\frac{|1+a_1|}{8}B(a_k(t)e^{iks})(|k|-\delta_1(k))\big\}\\
    &\quad+\sum_{k\neq\{0,1\}}\big\{\frac{(1+a_1)^2}{8|1+a_1|}B(\overline{a_k}(t)e^{-i(k-2)s})(|k|-\delta_1(k)-2\delta_2(k))\big\}\\
	&=\sum_{k\ne\{0,1\}}c_k(t)e^{iks}
	\end{split}
	\end{equation}
	where
	\begin{equation}\label{freq13}
 \begin{aligned}
c_k(t)&:=-\frac{A}{8}a_k(t)(2|k|+|k-1|-|k+1|)\\
 &\quad-\frac{B|1+a_1|}{8}\bigg[a_k(t)(|k|-\delta_1(k))+\frac{(1+a_1)^2}{|1+a_1|^2}\overline{a_{2-k}}(t)(|2-k|-\delta_1(2-k)-2\delta_2(2-k))\bigg].     
 \end{aligned}
\end{equation}

	We notice by identifying the coefficients of $e^{is}$ terms, that $c_0=c_1=0$.
	
	At the linearized level, the equation is
	\begin{equation}\label{freq14}
	\partial_ta_k(t)=c_k(t),\qquad k\in\Z.
	\end{equation}
	We can solve this system explicitly by noticing that the system decouples in pairs of linear equations. Indeed,
	
	$\bullet\,\,$ if $k\in\{0,1\}$ then $c_k(t)=0$, and the equation is simply $\partial_ta_k(t)=0$. This is consistent with the fact that that the functions $Y(s)=a_0+a_1e^{is}$ are stationary solutions of the original system.
	
	$\bullet\,\,$ if $k=2$ then the corresponding equation in the system \eqref{freq14} decouples,
	\begin{equation}\label{freq15}
	\partial_ta_2(t)=-\frac{A}{4}a_2(t)-\frac{B}{4}|1+a_1|a_2(t).
	\end{equation}
 The function $a_2$ has exponential decay as $t\to\infty$ if and only if $A+B|1+a_1|>0$. 
 Recalling that 
 \begin{equation*}
     \begin{aligned}
         A&=A(a_1)=T(|1+a_1|)=\frac{\mathcal{T}(|1+a_1|)}{|1+a_1|},\\
B&=B(a_1)=T'(|1+a_1|)=\frac{\mathcal{T}'(|1+a_1|)}{|1+a_1|}-\frac{\mathcal{T}(|1+a_1|)}{|1+a_1|^2},
     \end{aligned}
 \end{equation*} 
the structural condition \eqref{Tau_cond} guarantees that \begin{equation}\label{cond2}
    \begin{aligned}
A+B|1+a_1|=\mathcal{T}'(|1+a_1|)>0,        
    \end{aligned}
\end{equation}
thus $a_2(t)$ decays exponentially at a linear level.
	
	$\bullet\,\,$ if $k\geq 3$ then the system \eqref{freq14} decouples into systems of two coupled equations,
	\begin{equation}\label{freq16}
	\begin{split}
	\partial_ta_k(t)&=-\frac{1}{8}[(2k-2)A+kB|1+a_1|]a_k(t)+\frac{1}{8}(k-2)B\frac{(1+a_1)^2}{|1+a_1|}\overline{a_{2-k}}(t),\\
	\partial_t\overline{a_{2-k}}(t)&=-\frac{1}{8}[(2k-2)A+(k-2)B|1+a_1|]\overline{a_{2-k}}(t)+\frac{1}{8}kB\frac{(1+\overline{a_1})^2}{|1+a_1|}a_k(t).
	\end{split}
	\end{equation}
	These systems have exponentially decaying solutions if and only if all the eigenvalues of the associated matrices
	\begin{equation}\label{freq17}
	\begin{bmatrix}
	&(2k-2)A+kB|1+a_1| &-(k-2)B\frac{(1+a_1)^2}{|1+a_1|}\\
	&-kB\frac{(1+\overline{a_1})^2}{|1+a_1|} &(2k-2)A+(k-2)B|1+a_1|
	\end{bmatrix}
	\end{equation}
	have positive real parts. The eigenvalues are given by
 \begin{equation*}
     \begin{aligned}
         \lambda_1=2(k-1)(A+B|1+a_1|),\quad \lambda_2=2A(k-1),
     \end{aligned}
 \end{equation*}
 so condition \eqref{Tau_cond} guarantees their positivity.

	$\bullet\,\,$ if $k\leq -1$ then the system \eqref{freq14} decouples into systems of two coupled equations,
	\begin{equation}\label{freq18}
	\begin{split}
	\partial_ta_k(t)&=-\frac{1}{8}[(-2k+2)A-kB|1+a_1|]a_k(t)+\frac{1}{8}(2-k)B\frac{(1+a_1)^2}{|1+a_1|}\overline{a_{2-k}}(t),\\
	\partial_t\overline{a_{2-k}}(t)&=-\frac{1}{8}[(-2k+2)A+(2-k)B|1+a_1|]\overline{a_{2-k}}(t)-\frac{1}{8}kB\frac{(1+\overline{a_1})^2}{|1+a_1|}a_k(t).
	\end{split}
	\end{equation}
	This is equivalent to the system \eqref{freq16}, by letting $k=2-l$ and taking complex conjugates.
	
	\section{Wellposedness in the Wiener algebra}\label{sect:wiener}
	
	We begin by proving wellposedness in the Wiener Algebra space $W$, defined as in \eqref{W_norm}. We estimate the nonlinearity in these spaces. We will do this by decomposing the functions into Fourier series, and hence begin with the following lemma, which will then be used throughout the analysis.
	
	\begin{lemma}\label{lem:freq0}
		Let $A_k, B_k, C_k$ be given sequences, such that
		\begin{equation}\label{newB6}
		|C_n|\leq\sum_{k\in\Z}|A_k||B_{n-k}|.
		\end{equation}
		Then 
		\begin{equation}\label{newB6.61}
		\|\{C_k\}_{k\in\Z}\|_{N}\lesssim \|\{A_k\}_{k\in\Z}\|_{N}\times\|\{B_k\}_{k\in\Z}\|_{N}.
		\end{equation}
	\end{lemma}
	\begin{proof}
		Assume for simplicity that
		\begin{equation}\label{newB6.5}
		\|\{A_k\}_{k\in\Z}\|_{N}\leq 1\qquad\text{ and }\qquad\|\{B_k\}_{k\in\Z}\|_{N}\leq 1. 
		\end{equation}
		We would like to prove that 
		\begin{equation}\label{newB6.6}
		\|\{C_k\}_{k\in\Z}\|_{N}\lesssim 1.
		\end{equation}

		Indeed, we estimate first, for any $t\geq 0$,
		\begin{equation}\label{newB7}
		\sum_{n\in\Z}|C_n|\leq\sum_{n,k\in\Z}|A_k||B_{n-k}|\leq\sum_{l,k\in\Z}|A_k||B_l|\leq\Big(\sum_{k\in\Z}|A_k|\Big)\Big(\sum_{l\in\Z}|B_l|\Big)\leq 1.
		\end{equation}
		Also, for any $t\geq 0$ and $a\in \Z_+$ satisfying $2^at\geq 1$ we estimate
		\begin{equation}\label{newB8}
		(1+2^at)^{2/3}\sum_{|n|\in[2^{a-1},2^{a+1}]}|C_n|\leq (1+2^at)^{2/3}\sum_{|n|\in[2^{a-1},2^{a+1}]}\big[C^1_n+C^2_n+C^3_n\big],
		\end{equation}
		where, if $|n|\in[2^{a-1},2^{a+1}]$ we define
		\begin{equation*}
		\begin{split}
		&C_n^1:=\sum_{k\in\Z,\,|k|\leq 2^{a-4}}|A_k||B_{n-k}|,\\
		&C_n^2:=\sum_{k\in\Z,\,|n-k|\leq 2^{a-4}}|A_k||B_{n-k}|,\\
		&C_n^3:=\sum_{k\in\Z,\,|k|\geq 2^{a-6},\,|n-k|\geq 2^{a-6}}|A_k||B_{n-k}|.
		\end{split}
		\end{equation*}

		Then we estimate
		\begin{equation}\label{newB10}
		\begin{split}
		(1+2^at)^{2/3}\sum_{|n|\in[2^{a-1},2^{a+1}]}C^1_n&\lesssim (1+2^at)^{2/3}\sum_{|k|\leq 2^{a-4},\,|l|\in[2^{a-2},2^{a+2}]}|A_k||B_l|\\
		&\lesssim \Big(\sum_{|k|\leq 2^{a-4}}|A_k|\Big)(1+2^at)^{2/3}\Big(\sum_{|l|\in[2^{a-2},2^{a+2}]}|B_l|\Big)\\
		&\lesssim 1,
		\end{split}
		\end{equation}
		where we used the assumption \eqref{newB6.5} in the last bound. The bound on the contribution of $C^2_n$ is similar. Finally, for the High-High-to-Low interaction we can get better bounds (recall that $2^at\geq 1$)
		\begin{equation}\label{newB11}
		\begin{split}
		(1+2^at)^{4/3}&\sum_{|n|\in[2^{a-1},2^{a+1}]}C^3_n\lesssim (2^at)^{4/3}\sum_{b\geq a-8}\Big(\sum_{|k|\in[2^{b-10},2^{b+10}]}|A_k|\Big)\Big(\sum_{|l|\in[2^{b-10},2^{b+10}]}|B_l|\Big)\\
		&\lesssim \sum_{b\geq a-8}2^{\frac43(a-b)}(2^bt)^{4/3}\Big(\sum_{|k|\in[2^{b-10},2^{b+10}]}|A_k|\Big)\Big(\sum_{|l|\in[2^{b-10},2^{b+10}]}|B_l|\Big)\\
		&\lesssim 1,
		\end{split}
		\end{equation}
		where we used again the assumption \eqref{newB6.5} in the last bound. The desired estimate \eqref{newB6.6} follows from the bounds \eqref{newB7}--\eqref{newB11}.
	\end{proof}
	
	Before we proceed, we recall that
	\begin{equation}\label{freq18.1}
	\begin{split}
	\widetilde{Y}(s,s-\alpha)&=\frac{e^{-is}e^{-i\frac{\alpha}{2}}(Y(s-\alpha)-Y(s))}{2\sin(\frac{\alpha}{2})}\\&=\frac{e^{-is}e^{-i\frac{\alpha}{2}}\sum_{k\neq\{0,1\}}a_ke^{isk}(e^{-i\alpha k}-1)}{2\sin(\frac{\alpha}{2})},
	\end{split}
	\end{equation}
	\begin{lemma}\label{lem:freq1}
		Let $\widetilde{Y}(s,s-\alpha)$ be defined as in \eqref{freq18.1} and $Y\in W$. Then we have
		\begin{equation}\label{freq19}
		\widetilde{Y}^{n}(s,s-\alpha)=\sum_{k\in\Z}c_k^{(n)}(\alpha)e^{iks},
		\end{equation}
		for
		\begin{equation}\label{freq20}
		c_k^{(n)}(\alpha):=\sum_{p\in\Z}a_{p,k}^{(n)}e^{ip\alpha}, \qquad\text{where here}\quad \|\{b_k^{(n)}\}_{k\in\Z}\|_N\lesssim_{n}\|Y\|^{n}_W,
		\end{equation}
		with
		\begin{equation}\label{freq21}
		\sum_{p\in\Z}|a_{p,k}^{(n)}|=:b_k^{(n)}.
		\end{equation}
		Moreover, for $k=0,1$ and $n\geq2$, then
		\begin{equation}\label{freq21.1}
		\sup_{t\in[0,\infty)}b_k^{(n)}(1+t)^{4/3}\lesssim_{n}\|Y\|_W^{n}.
		\end{equation}
	\end{lemma}
	
	\begin{proof}
		We argue by induction. We begin by checking the case $n=1$. From \eqref{freq18.1}, we expand $(e^{-i\alpha k}-1)$ to get
		\begin{equation}\label{freq22}
		\begin{split}
		\widetilde{Y}(s,s-\alpha)&=\frac{e^{-is}e^{-i\frac{\alpha}{2}}\big(\sum_{k\geq2}a_ke^{isk}(e^{-i\alpha}-1)\sum_{m=0}^{k-1}e^{-im\alpha}+\sum_{k\leq-1}a_ke^{isk}(e^{i\alpha}-1)\sum_{m=0}^{|k|-1}e^{im\alpha}\big)}{2\sin(\frac{\alpha}{2})}\\&=-ie^{-is}e^{-i\alpha}\sum_{k\geq2}a_ke^{isk}\sum_{m=0}^{k-1}e^{-im\alpha}+ie^{-is}\sum_{k\leq-1}a_ke^{isk}\sum_{m=0}^{|k|-1}e^{im\alpha}.
		\end{split}
		\end{equation}
		We now define
		\begin{equation*}
		a_{p,k}:=\begin{cases}
		-ia_{k+1}\qquad&\text{for }p\in[-(k+1),-1],\quad k\geq1\\ ia_{k+1}\qquad&\text{for }p\in[0,|k+1|-1],\quad k\leq-2 \\0\qquad&\text{elsewhere}.
		\end{cases}
		\end{equation*}
		Then, we can rewrite \eqref{freq22} as
		\begin{equation*}
		\widetilde{Y}(s,s-\alpha)=\sum_{k\in\Z}c_k(\alpha)e^{isk}\qquad\text{where }c_k(\alpha):=\sum_{p\in\Z}a_{p,k}e^{ip\alpha}.
		\end{equation*}
		We now define
		\begin{equation*}
		b_k:=\sum_{p\in\Z}|a_{p,k}|.
		\end{equation*}
		Then we get
		\begin{equation*}
		b_k=|k+1||a_{k+1}|.
		\end{equation*}
		From the definition of the norms, the conclusion of the lemma holds for $n=1$. 
		
		We assume now that the conclusion is true for $\widetilde{Y}^{n}$, and show that it must hold for $\widetilde{Y}^{n+1}$. We have
		\begin{equation*}
		\widetilde{Y}^{n+1}(s,s-\alpha)=\sum_{m\in\Z}c_m^{(n+1)}(\alpha)e^{ims} =\sum_{k\in\Z}c_k^{(n)}(\alpha)e^{iks} \sum_{\ell\in\Z}c_\ell^{(1)}(\alpha)e^{i\ell s} 
		\end{equation*} 
		where
		\begin{equation*}
		c_m^{(n+1)}(\alpha)=\sum_{k+\ell=m}c_k^{(n)}(\alpha)c_\ell^{(1)}(\alpha),
		\end{equation*}
		with 
		\begin{equation*}
		c_k^{(n)}(\alpha):=\sum_{p_{n}\in\Z}a_{p_{n},k}^{(n)}e^{ip_{n}\alpha}\quad\text{with }b_k^{(n)}:=\sum_{p_{n}\in\Z}|a_{p_{n},k}^{(n)}|.
		\end{equation*}
		By induction assumption, $b_k^{(n)}$ satisfies \eqref{freq20}. Moreover, we have
		\begin{equation*}
		c_m^{(n+1)}(\alpha)=\sum_{p_1,p_{n}\in\Z}\sum_{k\in\Z}a_{p_{n},k}^{(n)}e^{ip_{n}\alpha}a^{(1)}_{p_1,m-k}e^{ip_1\alpha},
		\end{equation*}  
		and hence setting
		\begin{equation*}
		a_{p_,m}^{(n+1)}=\sum_{k\in\Z}\sum_{p_{n}\in\Z}a_{p_{n},k}^{(n)}a_{p-p_{n},m-k}^{(1)}
		\end{equation*}
		we sum over $p$ to get
		\begin{equation*}
		b_m^{(n+1)}\leq \sum_{k\in\Z}|b_k^{(n)}||b_{m-k}^{(1)}|.
		\end{equation*}
		Applying Lemma~\ref{lem:freq0} now concludes the proof of the lemma, with \eqref{freq21.1} being a consequence of \eqref{newB11}.
	\end{proof}
	
	\begin{lemma}\label{lem:freq2}
		Let $Y\in W$. Then we have
		\begin{equation}\label{freq23}
		\big(e^{-ir}Y'(r)\big)^n=\sum_{k\in\Z}b_k^{(n)}e^{ikr}\qquad\text{with}\quad\|\{b_k^{(n)}\}_{k\in\Z}\|_N\lesssim_n\|Y\|^n_W.
		\end{equation}
		Moreover, for $k=0,1$ and $n\geq2$, we have
		\begin{equation*}
		\sup_{t\in[0,\infty)}|b_k|^{(n)}(1+t)^{4/3}\lesssim_n\|Y\|_W^{n}.
		\end{equation*}
	\end{lemma}
	\begin{proof}
		We argue by induction. We begin by the looking at the case $n=1$. Taking $Y'(r)=\sum_{k\neq\{0\}}ika_ke^{ikr}$, we have
		\begin{equation*}
		e^{-ir}Y'(r)=\sum_{\ell\in\Z}b_\ell e^{ir\ell}\qquad\text{where}\quad b_\ell=i(k+1)a_{k+1},
		\end{equation*}
		and hence \eqref{freq23} clearly holds. We now assume that the conclusion holds for step $n$ and consider the step $n+1$. We get
		\begin{equation*}
		\big(e^{-ir}Y'(r) \big)^{n+1}=\sum_{m\in\Z}b_m^{(n+1)}e^{irm}=\sum_{k\in\Z}b_k^{(n)}e^{irk}\sum_{\ell\in\Z}b_\ell^{(1)}e^{ir\ell},
		\end{equation*}
		where here
		\begin{equation*}
		b_m^{(n+1)}=\sum_{k+\ell=m}b_k^{(n)}b_\ell^{(1)}.
		\end{equation*}
		We can now apply Lemma~\ref{lem:freq0} to conclude the proof.
	\end{proof}
	
	\begin{remark}\label{rem:freq1}
		We remark that analogues of Lemmas~\ref{lem:freq1}-\ref{lem:freq2} hold true for $\overline{\widetilde{Y}}(s,r)$ and $e^{ir}\overline{Y'}(r)$ respectively, by simply putting bars to the conclusions of the lemmas. Moreover, since $Y$ and $\overline{Y}$ have the same norm, the analogous lemmas for $\big[\widetilde{Y}(s,r)\big]^{m_1}\times\big[\overline{\widetilde{Y}}(s,r)\big]^{m_2}$and $\big[e^{-ir}Y'(r)\big]^{n_1}\times\big[e^{ir}\overline{Y'}(r)\big]^{n_2}$ are also proved in a completely identical way.
	\end{remark}
	
	\begin{lemma}\label{lem:freq3}
		Let $\widetilde{Y}(s,s-\alpha)$ be defined as in \eqref{freq18.1} and $Y\in W$. Then the nonlinear term $\mathcal{N}_g$ with $n_1+n_2+m_1+m_2\geq2$, defined as
		\begin{equation*}
		\mathcal{N}_g(s):=\int_{\T}e^{-i\alpha}\big[e^{-i(s-\alpha)}Y'(s-\alpha) \big]^{n_1}\big[e^{i(s-\alpha)}\overline{Y'}(s-\alpha) \big]^{n_2}\big[\widetilde{Y}(s,s-\alpha)\big]^{m_1}\big[\overline{\widetilde{Y}}(s,s-\alpha)\big]^{m_2}\frac{e^{is}e^{-i\frac{\alpha}{2}}}{2\sin(\frac{\alpha}{2})}d\alpha
		\end{equation*}
		can be written as
		\begin{equation*}
		\mathcal{N}_g(s)=\sum_{j\in\Z}B_je^{isj}\qquad\text{where}\quad\|\{B_j\}_{j\in\Z}\|_N\lesssim_{m+n}\|Y\|_W^{m+n}.
		\end{equation*} 
		where here we define $(m,n):=(m_1+m_2,n_1+n_2)$.
		Moreover, for $|j|<2$, we have
		\begin{equation*}
		\sup_{t\in[0,\infty)}|B_j|(1+t)^{4/3}\lesssim_{n+m}\|Y\|_W^{m+n}.
		\end{equation*}
	\end{lemma}
	\begin{proof}
		We use the decompositions into Fourier series from Lemmas~\ref{lem:freq1}-\ref{lem:freq2} and take Remark~\ref{rem:freq1} into consideration. By defining the pair $(m,n):=(m_1+m_2,n_1+n_2)$, we now consider the nonlinear term
		\begin{equation}\label{freq24}
		\begin{split}
		\mathcal{N}_g(s)&\!=\!\!\int_{\T}\!e^{-i\alpha}\big[e^{-i(s-\alpha)}Y'(s-\alpha) \big]^{n_1}\big[e^{i(s-\alpha)}\overline{Y'}(s-\alpha) \big]^{n_2}\big[\widetilde{Y}(s,s-\alpha)\big]^{m_1}\big[\overline{\widetilde{Y}}(s,s-\alpha)\big]^{m_2}\frac{e^{is}e^{-i\frac{\alpha}{2}}}{2\sin(\frac{\alpha}{2})}d\alpha\\
		&=\int_{\T}e^{-i\frac{3}{2}\alpha}\big(\sum_{k\in\Z}b_{k}^{(n)}e^{iks}\big)\big(\sum_{\ell\in\Z}c_{\ell}^{(m)}(\alpha)e^{i\ell s} \big)\frac{e^{is}}{2\sin(\frac{\alpha}{2})}d\alpha\\
		&=\sum_{k,\ell,p\in\Z}b_{k}^{(n)}a_{p,\ell}^{(m)}e^{i(k+\ell+1)s} \int_{\T}\frac{e^{-i\frac{\alpha}{2}}}{2\sin(\frac{\alpha}{2})}e^{-i(k+1-p)\alpha}d\alpha\\
		&=2\pi \sum_{k,\ell,p\in\Z}b_{k}^{(n)}a_{p,\ell}^{(m)}e^{i(k+\ell+1)s}\big(-\frac{i}{2}\sgn(k+1-p)\big),		
		\end{split}
		\end{equation}
		where we used Lemma~\ref{introLemma} in the last line. We now write $\mathcal{N}_g$ as
		\begin{equation*}
		\mathcal{N}_g(s)=\sum_{j\in\Z}B_je^{isj},
		\end{equation*}
		where for $j=k+\ell+1$ we set
		\begin{equation*}
		B_j=-i\pi\sum_{\substack{k,\ell,p\in\Z\\ j=k+\ell+1}}b_{k}^{(n)}a_{p,\ell}^{(m)}\sgn(k+1-p)
		\end{equation*}
		and define
		\begin{equation*}
		c_\ell:=\sum_{p\in\Z}|a_{p,\ell}|,
		\end{equation*}
		to get
		\begin{equation*}
		|B_j|\lesssim\sum_{k\in\Z}|b_{k}^{(n)}||c_{j-k-1}^{(m)}|.
		\end{equation*}
		An application of Lemma~\ref{lem:freq0} now concludes the proof.
	\end{proof}
	
	\begin{lemma}\label{lem:freq4}
		Let $\widetilde{Y}(s,s-\alpha)$ be defined as in \eqref{freq18.1}, let $Y\in W$ and $\mathcal{N}(t,s)$ be defined as in \eqref{freq7}. Then for $T$ analytic we have
		\begin{equation}\label{freq25}
		\mathcal{N}(t,s)-\mathcal{N}_0(t,s)-\mathcal{N}_1(t,s)=\sum_{k\in\Z}\mathcal{L}_k(t)e^{iks}\qquad\text{where}\qquad\|\{\mathcal{L}_k\}_{k\in\Z}\|_N\lesssim_2\|Y\|_W^2.
		\end{equation}
		Moreover, for $|k|<2$ we have
		\begin{equation}\label{freq26}
		\sup_{t\in[0,\infty)}|\mathcal{L}_k(t)|(1+t)^{4/3}\lesssim_{2}\|Y\|_W^2.
		\end{equation}
	\end{lemma}
	
	\begin{proof}
		After performing a Taylor expansion for the denominator and for the analytic function $T$ in \eqref{freq18.1}, the result follows immediately upon summing over $n_1,n_2,m_1$ and $m_2$ in Lemma~\ref{lem:freq3}.
	\end{proof}

	\section{Duhamel's formula and the fixed point argument}\label{sect:duhamel}
	We will apply Duhamel's formula to mode at a time. To this end, we denote by $\mathcal{N}^{(m)}$ the nonlinearity associated to the mode $m$.
	We hence get
	\begin{equation}\label{duhamel1}
	\partial_ta_m(t)-c_m(t)=\mathcal{L}_m(t),
	\end{equation}
	where $c_m(t)$ is defined in \eqref{freq13} and $\mathcal{L}_m$ as in Lemma~\ref{lem:freq4}.
	The idea is the following: for $m=0,1$, we show that the nonlinearity decays at a rate better than $1/t$. For $m\neq0,1$, we have less decay but we have the added decay from the parabolicity which we can take advantage of in order to close the argument.
	
	We now consider the various cases.

	\subsection{The modes $m=0,1$}
	For $m=0,1$, we already saw that $c_m(t)=0$. Hence we get
	\begin{equation*}
	\begin{split}
	&\partial_ta_0=\mathcal{L}_0(t)\\
	&\partial_ta_1=\mathcal{L}_1(t).
	\end{split}
	\end{equation*}
	This implies that
	\begin{equation*}
	\begin{split}
	&a_0(t)-a_0(0)=\int_{0}^{t}\mathcal{L}_0(\tau)\,d\tau,\\
	&a_1(t)-a_1(0)=\int_{0}^{t}\mathcal{L}_1(\tau)\,d\tau.
	\end{split}
	\end{equation*}
	From \eqref{freq26} in Lemma~\ref{lem:freq4}, we see that for $m=0,1$ we have
	\begin{equation}\label{first_freq}
	\sup_{t\in[0,\infty)}|a_m(t)-a_m(0)|\lesssim\int_{0}^t|\mathcal{L}_m(\tau)|\,d\tau\lesssim\int_{0}^t(1+\tau)^{-4/3}\,d\tau\lesssim\|Y\|_W^2.
	\end{equation}
 As a consequence,
 	\begin{equation*}
	\sup_{t\in[0,\infty)}|a_m(t)|\lesssim\varep+\|Y\|_W^2.
	\end{equation*}
In the following, we will denote $A(t)=A(a_1(t))$, $B(t)=B(a_1(t))$ and
\begin{equation}\label{A0B0}
    \begin{aligned}
    A_0=A(a_1(0))=T(|1+a_1(0)|),\qquad B_0=B(a_1(0))=T'(|1+a_1(0)|),
    \end{aligned}
\end{equation}
and also
\begin{equation}\label{A0B0}
    \begin{aligned}
    \widetilde{B}_0=|1+a_1(0)|B(a_1(0)),
    \end{aligned}
\end{equation}
and we notice that \eqref{first_freq} implies that 
\begin{equation}\label{AA0BB0}
    |A(t)-A_0|\lesssim \|Y\|_W^2,\quad |B(t)-B_0|\lesssim \|Y\|_W^2.
\end{equation}

	\subsection{The mode $m=2$}
	For $m=2$, we have 
	\begin{equation*}
	\partial_ta_2(t)=-\frac{A(t)+B(t)|1+a_1(t)|}{4}a_2(t)+\mathcal{L}_2(t).
	\end{equation*}
 Using \eqref{first_freq}-\eqref{AA0BB0}, we write
 	\begin{equation*}
	\partial_ta_2(t)=-\frac{A_0+\widetilde{B}_0}{4}a_2(t)+\widetilde{\mathcal{L}}_2(t),
	\end{equation*}
 with
 \begin{equation*}
	\widetilde{\mathcal{L}}_2(t)=\mathcal{L}_2(t)+\frac{A_0-A(t)+B_0|1+a_1(0)|-B(t)|1+a_1(t)|}{4}a_2(t).
	\end{equation*}

 Estimates \eqref{first_freq} and \eqref{AA0BB0} imply that $\widetilde{\mathcal{L}}_2$ is comparable in size to $\mathcal{L}_2$. Moreover, $A_0+\widetilde{B}_0>0$ follows from the natural conditions \eqref{Tau_cond}, as in  \eqref{cond2}.
From Duhamel's formula, we get
	\begin{equation*}
	a_2(t)=e^{-\frac{1}{4}(A_0+\widetilde{B}_0)t}a_2(0)+\int_{0}^te^{-\frac{1}{4}(t-\tau)(A_0+\widetilde{B}_0)}\widetilde{\mathcal{L}}_2(\tau)\,d\tau
	\end{equation*}

	\subsection{The modes $m\geq3$}
	For $m\geq3$, we get a system of two coupled equations.
	\begin{equation}\label{duhamel2}
	\begin{split}
	\partial_ta_m(t)&=-\frac{1}{8}[2(m-1)A_0+m\widetilde{B}_0]a_m(t)+\frac{1}{8}(m-2)B_0\frac{(1+a_1(0))^2}{|1+a_1(0)|}\overline{a_{2-m}}(t)+\widetilde{\mathcal{L}}_m(t),\\
	\partial_t\overline{a_{2-m}}(t)&=-\frac{1}{8}[2(m-1)A_0+(m-2)\widetilde{B}_0]\overline{a_{2-m}}(t)+\frac{1}{8}mB_0\frac{(1+\overline{a_1}(0))^2}{|1+a_1(0)|}a_m(t)+\widetilde{\overline{\mathcal{L}}}_{2-m}(t),
	\end{split}
	\end{equation}
 where
 	\begin{equation*}
	\begin{split}
	\widetilde{\mathcal{L}}_m(t)&:=\mathcal{L}_m(t)-\frac{2(m-1)(A(t)-A_0)+m(B|1+a_1(t)|-\widetilde{B}_0)}{8}a_m(t)\\
 &\quad+\frac18(m-2)\Big(B(t)\frac{(1+a_1(t))^2}{|1+a_1(t)|}-B_0\frac{(1+a_1(0))^2}{|1+a_1(0)|}\Big)\overline{a_{2-m}}(t),\\
	\widetilde{\overline{\mathcal{L}}}_{2-m}(t)&:=\overline{\mathcal{L}_{2-m}}(t)+\frac{1}{8}m\Big(B(t)\frac{(1+\overline{a_1}(t))^2}{|1+a_1(t)|}-B_0\frac{(1+\overline{a_1}(0))^2}{|1+a_1(0)|}\Big)a_m(t)\\
 &\quad-\frac{1}{8}\Big(2(m-1)(A(t)-A_0)+(m-2)\big(B(t)|1+a_1(t)|-\widetilde{B}_0\big)\Big)\overline{a_{2-m}}(t).
	\end{split}
	\end{equation*}
 
	We denote by $\mathcal{G}$ the matrix
	\begin{equation*}
	\mathcal{G}=-\frac{1}{8}\begin{pmatrix}
	&(2m-2)A_0+m\widetilde{B}_0 &-(m-2)B_0\frac{(1+a_1(0))^2}{|1+a_1(0)|}\\
	&-mB_0\frac{(1+\overline{a_1}(0))^2}{|1+a_1(0)|} &(2m-2)A_0+(m-2)\widetilde{B}_0
	\end{pmatrix},
	\end{equation*}
	and get 
	\begin{equation}\label{eq1}
	\begin{pmatrix}
	a_m(t) \\ \overline{a_{2-m}}(t)
	\end{pmatrix}=e^{\mathcal{G}t}\begin{pmatrix}
	a_m(0) \\ \overline{a_{2-m}}(0)
	\end{pmatrix}+\int_{0}^te^{(t-\tau)\mathcal{G}}\begin{pmatrix}
	\widetilde{\mathcal{L}}_m(\tau) \\ \widetilde{\overline{\mathcal{L}}}_{2-m}(\tau)
	\end{pmatrix}d\tau.
	\end{equation}
	The two eigenvalues of $-8\mathcal{G}$ we find are
	\begin{equation*}
	\lambda_1=2A_0(m-1)\qquad\text{and}\qquad\lambda_2=2(A_0+\widetilde{B}_0)(m-1),
	\end{equation*}
	which are positive for any $m\geq3$ by \eqref{Tau_cond}. Moreover, the matrix $\mathcal{G}$ is hermitian, hence it can be diagonalized by a unitary matrix. Therefore, if we define for $m\geq3$
	\begin{equation*}
	\mathcal{C}_m(t):=\int_{0}^te^{(t-\tau)\mathcal{G}}\begin{pmatrix}
	\widetilde{\mathcal{L}}_m(\tau) \\ \widetilde{\overline{\mathcal{L}}}_{2-m}(\tau)
	\end{pmatrix}d\tau,
	\end{equation*} 
	it follows that
	\begin{equation*}
 \begin{aligned}
     	|\mathcal{C}_m(t)|&\leq \int_{0}^t\bigg|e^{\mathcal{G}(t-\tau)}\begin{pmatrix}
	\widetilde{\mathcal{L}}_m(\tau) \\ \widetilde{\overline{\mathcal{L}}}_{2-m}(\tau)
	\end{pmatrix} \bigg|d\tau\leq\int_{0}^t e^{-\frac{1}{8}\min\{\lambda_1,\lambda_2\}(m-1)(t-\tau)}\Big|\begin{pmatrix}
	\widetilde{\mathcal{L}}_m(\tau) \\ \widetilde{\overline{\mathcal{L}}}_{2-m}(\tau)
	\end{pmatrix}\Big|d\tau,
 \end{aligned}
	\end{equation*}
	Defining $\widetilde{\mathcal{L}}_k(\tau)=0$ for $k=0,1$, we have
	\begin{equation}\label{newB1}
	\begin{split}
	\|\mathcal{C}\|_{W}&\lesssim \sup_{t\in[0,\infty)}\int_0^t\sum_{k\in\Z}|k|e^{-C_0|k|(t-\tau)}|\widetilde{\mathcal{L}}_k(\tau)|\,d\tau\\
	&+\sup_{t\in[0,\infty),\,m\in\Z_+}\int_0^t\sum_{|k|\in[2^{m-1},2^{m+1}]}(1+|k|t)^{2/3}|k|e^{-C_0|k|(t-\tau)}|\widetilde{\mathcal{L}}_k(\tau)|\,d\tau,
	\end{split}
	\end{equation}
 where $C_0=C_0(A_0, \widetilde{B}_0)$, which in particular gives 
 \begin{equation}\label{newB1}
	\begin{split}
	\|\mathcal{C}\|_{W}&\lesssim \sup_{t\in[0,\infty)}\int_0^t\sum_{k\in\Z}|k|(1+(t-\tau)|k|)^{-4}|\widetilde{\mathcal{L}}_k(\tau)|\,d\tau\\
	&+\sup_{t\in[0,\infty),\,m\in\Z_+}(1+2^mt)^{2/3}\int_0^t\sum_{|k|\in[2^{m-1},2^{m+1}]}|k|(1+(t-\tau)|k|)^{-4}|\widetilde{\mathcal{L}}_k(\tau)|\,d\tau.
	\end{split}
	\end{equation}
	Assuming that $\|\widetilde{\mathcal{L}}_k\|_{N}\leq 1$, we would have
	\begin{equation}\label{newB2}
	\begin{split}
	\sum_{k\in\Z}|\widetilde{\mathcal{L}}_k(t)|&\lesssim 1\qquad\text{ for any }t\in[0,\infty),\\
	\sum_{|k|\in[2^{m-1},2^{m+1}]}|\widetilde{\mathcal{L}}_k(t)|&\lesssim (1+2^mt)^{-2/3}\qquad\text{ for any }t\in[0,\infty)\text{ and }m\in\Z_+.
	\end{split}
	\end{equation} 
	It follows from \eqref{newB1} that 
     \begin{equation}\label{newB3}
	\begin{split}
	\|\mathcal{C}\|_{W}&\lesssim \sup_{t\in[0,\infty)}\sum_{a\in\Z_+}2^a\int_0^t(1+(t-\tau)2^a)^{-4}\sum_{|k|\in[2^{a-1},2^{a+1}]}|\widetilde{\mathcal{L}}_k(\tau)|\,d\tau\\
	&+\sup_{t\in[0,\infty),\,m\in\Z_+}(1+2^mt)^{2/3}2^m\int_0^t(1+(t-\tau)2^m)^{-4}\sum_{|k|\in[2^{m-1},2^{m+1}]}|\widetilde{\mathcal{L}}_k(\tau)|\,d\tau
	\\&\lesssim1,
	\end{split}
	\end{equation}
	upon using \eqref{newB2} and considering the usual cases $2^m\leq 1/t$, $2^m\geq 1/t$, $2^a\leq 1/t$, $2^a\geq 1/t$.
	
	Therefore, from \eqref{eq1}, upon summing over $m$ and applying the $W$-norm we see that
	\begin{equation*}
 \|Y\|_W\lesssim\|Y_0\|_W+\|Y\|_W^2\lesssim\varep+\varep^2,
	\end{equation*}
	where here $Y_0$ denotes the free evolution, thus closing the argument.
	
	\section{Wellposedness in space allowing for initial data with corners}\label{sect:corners}
	We now prove wellposedness in a space whose initial data allows for corners. We will use the norms $Z_1$ and $Z_2$ as defined in \eqref{norm2} and \eqref{norm3} respectively.
	
	\subsection{The $Z_1$ norm is an algebra}
	The first case we consider is both functions being of $Y'$-type. We study the model nonlinearity 
	\begin{equation*}
	\mathcal{Q}(Y_1,Y_2)(s,t):=\int_{\T}\frac{e^{is}e^{-i\alpha/2}}{2\sin(\frac{\alpha}{2})}\big[e^{-i(s-\alpha)}Y_1'(s-\alpha)\big]\big[e^{-i(s-\alpha)}Y_2'(s-\alpha)\big]\,d\alpha.
	\end{equation*}
	to get the following lemma.
	\begin{lemma}\label{lem:kernel1}
		Let $Y_1,Y_2\in Z_1$. Then 
		\begin{equation}\label{norm4}
		\|(Y'_1\cdot Y'_2)(t)\|_{Z_1}\lesssim\|Y_1\|_{Z_1}\|Y_2\|_{Z_1}.
		\end{equation}
		Moreover,
		\begin{equation}\label{norm5}
		\sup_{t\geq 0}\sup_{k\in\Z_+}2^{k/2}(1+2^kt)^{2/3}\|P_k[\mathcal{Q}(Y_1,Y_2)(t)]\|_{L^2}\lesssim\|Y_1\|_{Z_1}\|Y_2\|_{Z_1}.
		\end{equation}
	\end{lemma}
	\begin{proof}
		Assume $t\in[0,\infty)$ is fixed. Then
		\begin{equation*}
		(1+t)^{4/3}\|(Y'_1\cdot Y_2')(t)\|_{L^\infty}\lesssim(1+t)^{2/3}\|Y_1'(t)\|_{L^\infty}(1+t)^{2/3}\|Y_2'(t)\|_{L^\infty}\lesssim\|Y_1\|_{Z_1}\|Y_2\|_{Z_1}.
		\end{equation*}
		Moreover, for any $k\in\Z_+$,
		\begin{equation*}
		\begin{split}
		P_k(Y_1'\cdot Y_2')&=P_k(P_{\leq k-4}Y_1'\cdot P_{[k-2,k+2]}Y_2')+P_k(P_{\leq k-4}Y_2'\cdot P_{[k-2,k+2]}Y_1')\\
		&+\sum_{k_1,k_2\geq k-3, |k_1-k_2|\leq6}P_k(P_{k_1}Y_1'\cdot P_{k_2}Y_2').
		\end{split}
		\end{equation*}
		We now estimate, for $t\in[0,\infty)$ and $k\in\Z_+$,
		\begin{equation*}
		\begin{split}
		&2^{k/2}(1+2^kt)^{2/3}\|P_k(Y_1'\cdot Y_2')(t)\|_{L^2}\\	
		&\lesssim2^{k/2}(1+2^kt)^{2/3}\Big\{2^k\|P_{\leq k-4}Y_1'(t)\|_{L^\infty}\|P_{[k-2,k+2]}Y_2(t)\|_{L^2} +2^k\|P_{\leq k-4}Y_2'(t)\|_{L^\infty}\|P_{[k-2,k+2]}Y_1(t)\|_{L^2}
		\\&\qquad +\sum_{k_1,k_2\geq k-3, |k_1-k_2|\leq6}2^{k/2}2^{k_1}\|P_{k_1}Y_1(t)\|_{L^2}2^{k_2}\|P_{k_2}Y_2(t)\|_{L^2}
		\Big\}\\
		&\lesssim \|Y_1\|_{Z_1}\|Y_2\|_{Z_1}+\sum_{k_1,k_2\geq k-3, |k_1-k_2|\leq6}2^k(1+2^kt)^{2/3}2^{-k_1}(1+2^{k_1}t)^{-4/3}\|Y_1\|_{Z_1}\|Y_2\|_{Z_1}\\
		&\lesssim \|Y_1\|_{Z_1}\|Y_2\|_{Z_1}.
		\end{split}
		\end{equation*}
		
		The bounds \eqref{norm5} follow from \eqref{norm4} and the $L^2$ boundedness of the Hilbert transform. 
		
	\end{proof}
\begin{remark}
    Lemma \ref{lem:kernel1} implies that by means of an easy induction argument we can write $\textbf{Y}':=Y_1'\times\dots\times Y_n'$ and $\textbf{Y}'$ will act as a single copy of $Y'$.
\end{remark}
	
	\subsection{Reformulation as a multilinear paraproduct}
	We need to estimate a nonlinearity of the following form
	\begin{equation}\label{hl1}
	\begin{split}
	\mathcal{P}(g_1,\cdots,g_{m},f_1,\cdots,f_{n})(s)&=\int_{\T}\widetilde{g}_1(s,s-\alpha)\cdots \widetilde{g}_m(s,s-\alpha)A(\alpha)\\&\qquad\times f'_1(s-\alpha)\cdots f'_n(s-\alpha)\frac{e^{is}e^{-i\alpha/2}}{2\sin(\alpha/2)}d\alpha.
	\end{split}
	\end{equation}
	where here the kernel $A(\alpha)$ is defined by
	\begin{equation}\label{hl3}
	A(\alpha):=e^{-ir\alpha}\qquad\text{for some }r\in\Z.
	\end{equation}
	The reasoning behind introducing $A(\alpha)$ is to allow for conjugates as well in \eqref{hl1}.
	
	We now define the function
	\begin{equation}\label{H}
	H(s-\alpha):=f_1'(s-\alpha)\cdots f_n'(s-\alpha).
	\end{equation}
	We know from Lemma~\ref{lem:kernel1} that since the $Z_1$ space is an algebra, that $H$ acts as though it were only a single copy $f'$. Moreover, we recall that we can rewrite
	\begin{equation}\label{hl4}
	\widetilde{g}_i(s,s-\alpha)=\frac{e^{-is}e^{-i\alpha/2}\sum_{k_i\neq\{0,1\}}a^{(i)}_{k_i}e^{isk_i}(e^{-i\alpha k_i}-1)}{2\sin(\alpha/2)}.
	\end{equation}
	Hence, we can reformulate \eqref{hl1} as
	\begin{equation}\label{hl5}
	\begin{split}
	\mathcal{P}(g_1,\cdots,g_m,&f_1,\cdots,f_n)(s)=\int_{\T}\frac{e^{is}e^{-i\alpha/2}}{2\sin(\frac{\alpha}{2})}A(\alpha)
	\frac{e^{-is}e^{-i\alpha/2}\sum_{k_1\neq\{0,1\}}\big(a^{(1)}_{k_1}e^{isk_1}(e^{-i\alpha k_1}-1)\big)}{2\sin(\alpha/2)}\times\dots\\
 &\qquad\times\frac{e^{-is}e^{-i\alpha/2}\sum_{k_m\neq\{0,1\}}\big(a^{(n)}_{k_m}e^{isk_m}(e^{-i\alpha k_m}-1)\big)}{2\sin(\alpha/2)} H(s-\alpha)d\alpha\\
	&=\int_{\T}\frac{e^{-ims}e^{-i(m+1)\alpha/2}}{2^{m+1}\sin^{m+1}(\alpha/2)}A(\alpha)H(s-\alpha)\\&\quad\times\sum_{k_1,\cdots,k_m\neq\{0,1\}}a^{(1)}_{k_1}\cdots a^{(m)}_{k_m}e^{is(k_1+\cdots+k_m+1)}(e^{-i\alpha k_1}-1)\cdots(e^{-i\alpha k_m}-1)\,d\alpha. 
	\end{split}
	\end{equation}
	By writing 
	\begin{equation}
	a_{k_i}=\frac{1}{2\pi}\int_{\T}g_i(y_i)e^{-iy_ik_i}dy_i,
	\end{equation}
	we can rewrite the model nonlinearity in physical space to get
	\begin{equation}\label{hl6}
	\begin{split}
	\mathcal{P}(g_1,\cdots,g_m,f_1,\cdots,f_n)(s)&=\frac{1}{(2\pi)^m}\int_{\T}\int_{\T^m}\sum_{k_1,\cdots,k_m\neq\{0,1\}}e^{-ims}e^{i((s-y_1)k_1+\cdots+(s-y_m)k_m+s)}\\
	&\qquad\times(e^{-i\alpha k_1}-1)\cdots(e^{-i\alpha k_m}-1)\frac{e^{-i(m+1)\alpha/2}}{2^{m+1}\sin^{m+1}(\alpha/2)}\\
	&\qquad\times g_1(y_1)\cdots g_m(y_m)A(\alpha)H(s-\alpha)\,dy_1\cdots dy_md\alpha. 
	\end{split}
	\end{equation}
	Upon performing the split
	\begin{equation}
	H(s-\alpha)=P_{\leq k-4m}H(s-\alpha)+ P_{\geq k-4m}H(s-\alpha),
	\end{equation}
	and analyzing the nonlinearity in conjunction with the Littlewood-Paley projections yields
	\begin{equation}\label{hl7}
	\begin{split}
\mathcal{P}_k(P_{k_1}g_1,\cdots,&P_{k_m}g_m,P_{k_{m+1}}f_1,\cdots,P_{k_{n+m}}f_n)(s)\\
 &=\frac{1}{(2\pi)^m}P_k\sum_{k_1,\dots,k_m\in\mathbb{Z}}\int_{\T}\bigg(\int_{\T^m}\frac{e^{is}e^{-ims}e^{-i\alpha/2}}{2\sin(\alpha/2)} P_{k_1}g_1(y_1)\cdots P_{k_m}g_m(y_m)\\
 &\quad\times 
	L_{k_1}(s-y_1,\alpha)\cdots L_{k_m}(s-y_m,\alpha)A(\alpha)\,dy_1\cdots dy_m\\
 &\quad\times \big[P_{\leq k-4m}H(s-\alpha)+ P_{\geq k-4m}H(s-\alpha)\big]\bigg)\,d\alpha, 
	\end{split}
	\end{equation}
	where here
	\begin{equation}\label{Ldef}
	L_n(s,\alpha):=\sum_{k\neq\{0,1\}}\varphi_{n+2}(k)\frac{e^{-i\alpha/2}(1-e^{-i\alpha k})}{2\sin(\alpha/2)}e^{isk}.
	\end{equation}

	\subsection{Analysis of the kernels}
	In order to get a better understanding of the the paraproduct, we begin by studying the kernels.
  We first rewrite $L_n$ in a more convenient form. We will use the Poisson summation formula which tells us that if $f$ is a Schwartz function, then 
		\begin{equation}
		\sum_{k\in\Z}f(k)=\sum_{k\in\Z}\widehat{f}(2\pi k).
		\end{equation}
		In our case $f(x)=\varphi_n(x)e^{is x}(1-e^{-i\alpha x})$. Hence we have
		\begin{equation}
		\begin{split}
		\widehat{f}(2\pi k)&=\int_{\R}f(x)e^{-i2\pi k x}dx\\
		&=\int_{\R}\varphi_n(x)e^{is x}(1-e^{-i\alpha x})e^{-i2\pi kx}\,dx\\
		&=2^n\int_{\R}\Big[\varphi_0(y)e^{i2^ny(s-2\pi k)}-\varphi_0(y)e^{i2^ny(s-2\pi k-\alpha)} \Big]dy\\
		&=:2^n\Big[\phi_0(2^n(s-2\pi k))-\phi_0(2^n(s-2\pi k-\alpha)) \Big],
		\end{split}
		\end{equation}
  where we remark that $\phi_0$ is a Schwartz function.
Let us denote
\begin{equation}\label{psi_def}
    \begin{aligned}
    \psi_n(s):=\sum_{k\neq1}\varphi_{n+2}(k)e^{isk}=\sum_{k\in\mathbb{Z}}\varphi_{n+2}(k)e^{isk}-\varphi_{n+2}(1)e^{is}.
    \end{aligned}
\end{equation}
Thus,  for $k\neq1$, $\widehat{\psi}_n(k)=\varphi_{n+2}(k)$.
By the Poisson formula,
\begin{equation*}
    \begin{aligned}
        \psi_n(s)=2^{n+2}\sum_{k\in\mathbb{Z}}\phi_0(2^{n+2}(s-2\pi k))-\varphi_{n+2}(1)e^{is}.
    \end{aligned}
\end{equation*}
Then
\begin{equation}\label{Ln_psi}
    \begin{aligned}
        L_n(s,\alpha)&=\frac{e^{-i\alpha/2}}{2 \sin(\alpha/2)}\big(\psi_n(s)-\psi_n(s-\alpha)\big),
    \end{aligned}
\end{equation}
where 
\begin{equation*}
    \begin{aligned}
        \int_{\mathbb{T}}|\psi_n(s)|ds&\lesssim 2^n\sum_{k\in\mathbb{Z}}\int_{\mathbb{T}}|\phi_0(2^{n+2}(s-2\pi k))|ds\\
        &=2^n\sum_{k\in\mathbb{Z}}\int_{-2\pi k}^{2\pi(1-k)}|\phi_0(2^{n+2} w)|dw =\int_{\mathbb{R}}|\phi_0(y)|dy\lesssim1,
    \end{aligned}
\end{equation*}
and similarly for $m\geq0$,
\begin{equation}\label{psi_bounds}
    \begin{aligned}
        \int_{\mathbb{T}}|\psi_n^{m)}(s)|ds\lesssim 2^{mn}.
    \end{aligned}
\end{equation}
The expression \eqref{Ln_psi} and the integrability bound \eqref{psi_bounds} will be frequently use throughout the paper.

 Next, we get the following two lemmas for the kernels.
	\begin{lemma}\label{lem:kernel_1}
		The following bound holds:
		\begin{equation}
		\int_{\T}|L_n(s,\alpha)|ds\lesssim\min\{2^n,|\alpha|^{-1} \}.
		\end{equation}
	\end{lemma}

	\begin{proof}
	
 We have that, for $2^n|\alpha|<1$,
		\begin{equation}\begin{aligned}
		    \int_{\T}|L_n(s,\alpha)|ds&=\int_{\T}\bigg|\frac{e^{-i\alpha/2}}{2 \sin(\alpha/2)}\big(\psi_n(s)-\psi_n(s-\alpha)\big)\bigg|ds\\
      &\leq\frac{|\alpha|}{|2\sin(\alpha/2)|}\int_0^1\int_{\T}|\psi_n'(s-\alpha(z-1))|dsdz\lesssim2^n,
		\end{aligned}
		\end{equation}
		while that for $2^n|\alpha|\geq1$, we write
		\begin{equation}
		\begin{split}
		\int_{\T}\bigg|\frac{e^{-i\alpha/2}}{2 \sin(\alpha/2)}\big(\psi_n(s)-\psi_n(s-\alpha)\big)\bigg|ds&\leq\bigg|\frac{e^{-i\alpha/2}}{2 \sin(\alpha/2)}\bigg|\int_{\T}\big(|\psi_n(s)|+|\psi_n(s-\alpha)|\big)ds\\
		&\lesssim |\alpha|^{-1},
		\end{split}
		\end{equation}
  where in both cases  we used the estimate \eqref{psi_bounds} in the last step.
	\end{proof}
	
	These bounds are not enough and lead to logarithmic losses, hence we introduce the modified kernel $\widetilde{L}_n$ defined by
	\begin{equation}\label{Ltilde}
	\widetilde{L}_n(s,\alpha):=L_n(s,\alpha)-\frac{e^{-i\alpha/2}}{2\sin(\alpha/2)}\min\{1,2^n\alpha\}2^{-n}\psi_n'(s).
	\end{equation}
 
	Then we get the following bounds.
	
	\begin{lemma}\label{lem:kernel2}
		The following bounds hold:
		\begin{equation}
		\begin{split}
		\int_{\T}|\widetilde{L}_n(s,\alpha)|ds\lesssim\min\{2^{2n}|\alpha| ,|\alpha|^{-1}\}
		\end{split}
		\end{equation}
		Moreover, we have
		\begin{equation}\label{tildebound}
		\int_{\T}\bigg|\frac{d}{d\alpha}\widetilde{L}_n(s,\alpha) \bigg|ds\lesssim\min\{2^{2n},|\alpha|^{-2}\}.
		\end{equation}
	\end{lemma}
	
	\begin{proof}
		For $2^n|\alpha|\geq1$, the bounds simply follow from the definition of $\widetilde{L}_n$, Lemma~\ref{lem:kernel_1} and the integrability bound on $\psi_n'$ \eqref{psi_bounds}. For $2^n|\alpha|\leq1$, we have
		\begin{equation*}
		\widetilde{L}_n(s,\alpha)=\frac{e^{-i\alpha/2}}{2\sin(\alpha/2)}\big(\psi_n(s)-\psi_n(s-\alpha)-\alpha\psi_n'(s)\big),
		\end{equation*}
		which yields
		\begin{equation}
		\begin{split}
		\bigg|\frac{e^{-i\alpha/2}}{2\sin(\alpha/2)}\bigg|\int_{\T}\big|\psi_n(s)-\psi_n(s-\alpha)&-\alpha\psi_n'(s)\big|ds\\
		&\leq\frac{1}{|2\sin(\alpha/2)|}\int_{\mathbb{T}}\bigg|\int_{s-\alpha}^{s}\big(\psi_n'(y)-\psi_n'(s) \big)dy \bigg|ds\\
		&\leq \frac{1}{2|\sin(\alpha/2)|}\int_{\T}\int_{s-\alpha}^{s}\int_{s}^{y}|\psi_n''(q)|\,dqdyds\\
		&\lesssim2^{2n}|\alpha|.
		\end{split}
		\end{equation}
		Finally, we calculate 
		\begin{equation}
		\begin{split}
		\frac{d}{d\alpha}\widetilde{L}_n(s,\alpha)&=\frac{-i}{2}\frac{e^{-i\alpha/2}}{2\sin(\alpha/2)}\big(\psi_n(s)-\psi_n(s-\alpha)-\alpha\psi_n'(s)\big)\\
		&+\frac{e^{-i\alpha/2}}{2\sin(\alpha/2)}\big(\psi_n'(s-\alpha)-\psi_n'(s)\big)\\
		&-\frac{e^{-i\alpha/2}\cos(\alpha/2)}{4\sin^2(\alpha/2)}\big(\psi_n(s)-\psi_n(s-\alpha)-\alpha\psi_n'(s))\big)
		\end{split}
		\end{equation}
		which by similar argument as above for $2^n|\alpha|<1$ and direct estimation for $2^n|\alpha|>1$ yields
		\begin{equation}
		\int_{\T}\bigg|\frac{d}{d\alpha}\widetilde{L}_n(s,\alpha)\bigg|ds\lesssim\min\{{2^{2n},|\alpha|^{-2}}\}.
		\end{equation}	
	\end{proof}
	
	\section{Useful estimates}\label{sect:estimates}
	We use this section to collect estimates which will serve as the building blocks for the argument. We first state the following identity for $g\in Z_1$ and $k\in\mathbb{N}$
	\begin{equation}\label{identity1}
 P_{k}g(s)=\int_{\T}P_{k}g(y)\psi_k(s-y)\,dy,
	\end{equation}
which follows trivially since $\widehat{\psi}_k(m)=\varphi_{k+2}(m)$.
This identity will be used throughout the paper.
	\begin{lemma}\label{lem_est0} 
		Let $g\in Z_1$, $k\in\mathbb{N}$ and $L_{\leq k}(s-y,\alpha)$ be defined as in \eqref{L_sum}. Then we have
		\begin{equation}\label{est0}
		\bigg\|\int_{\T}P_{\leq k}g(y)L_{\leq k}(s-y,\alpha)dy\bigg\|_{L^\infty_s}\lesssim\|P_{\leq k}g'\|_{L^\infty}.
		\end{equation}
	\end{lemma}
   
	\begin{proof}
 We split the terms in $L_{\leq k}$,
 \begin{equation*}
    \begin{aligned}
		\int_{\T}P_{\leq k}g(y)L_{\leq k}(s-y,\alpha)dy&=\sum_{b\leq k}\int_{\T}P_{b}g(y)\frac{e^{-i\alpha/2}}{2\sin{(\alpha/2)}}\big(\psi_n(s-y)-\psi_n(s-y-\alpha)\big)dy\\
  &=\sum_{b\leq k}e^{-i\alpha/2}\frac{P_{b}g(s)-P_{b}g(s-\alpha)}{2\sin{(\alpha/2)}}=e^{-i\alpha/2}\frac{P_{\leq k}g(s)-P_{\leq k}g(s-\alpha)}{2\sin{(\alpha/2)}},
  \end{aligned}
\end{equation*}
hence
    \begin{equation*}
    \begin{aligned}
    \big|\big|\int_{\T}P_{\leq k}g(y)L_{\leq k}(s-y,\alpha)dy\big|\big|_{L^\infty_s}&\lesssim \|P_{\leq k}g'\|_{L^\infty}.
  \end{aligned}
\end{equation*}

	\end{proof}

When $2^k|\alpha|\leq1$, Lemma \ref{lem_est1} will not be enough to control logarithmic losses. We need to isolate the dependence on $\alpha$ to use cancellations given by the oddness of the kernels.

 	\begin{lemma}\label{lem_est1}
		Let $g\in Z_1$, $k\in\mathbb{N}$ and $L_{\leq k}(s-y,\alpha)$ be defined as in \eqref{L_sum}. Then, when $2^k|\alpha|\leq1$, we have
  
		\begin{equation}\label{est0}
		\bigg\|\int_{\T}P_{\leq k}g(y)L_{\leq k}(s-y,\alpha)dy-\frac{e^{-i\alpha/2}\alpha}{2\sin(\alpha/2)}P_{\leq k}g'(s)\bigg\|_{L^\infty_s}\lesssim2^{k}|\alpha|\|g'\|_{L^\infty}.
		\end{equation}
	\end{lemma}
  
	\begin{proof}
		Introducing the correction $\widetilde{L}_n$ \eqref{Ltilde} with $2^k|\alpha|\leq1$, we begin by rewriting
		\begin{equation}\label{est1}
		\begin{split}
		\int_{\T}P_{\leq k}g(y)L_{\leq k}(s-y,\alpha)dy&=\sum_{b\leq k}\int_{\T}P_{b}g(y)L_{b}(s-y,\alpha)dy\\
		&=\sum_{b\leq k}\int_{\T}P_{b}g(y)\bigg(\widetilde{L}_{b}(s-y,\alpha)+\frac{e^{-i\alpha/2}\alpha}{2\sin(\alpha/2)}\psi_n'(s-y)\bigg)dy\\
		&=\frac{e^{-i\alpha/2}\alpha}{2\sin(\alpha/2)}P_{\leq k}g'(s)+\sum_{b\leq k}\int_{\T}P_{b}g(y)\widetilde{L}_{b}(s-y,\alpha)dy,
		\end{split}
		\end{equation}
		where we used an integration by parts in $y$ and the identity \eqref{identity1} in the last line. The bound \eqref{est0} now follow immediately from an application of Lemma~\ref{lem:kernel2},
			\begin{equation}\label{est1}
		\begin{split}
		\|\sum_{b\leq k}\int_{\T}P_{b}g(y)\widetilde{L}_{b}(s-y,\alpha)dy\|_{L^\infty_s}\lesssim \sum_{b\leq k}2^{2b}|\alpha|\|P_{b}g\|_{L^\infty}\lesssim 2^k|\alpha|\sup_{b}\|P_{b}g\|_{L^\infty}.
		\end{split}
		\end{equation}
	\end{proof}

	\begin{lemma}\label{lem_est2}
		Let $f\in Z_1$, $k\in\mathbb{N}$ and the function $H$ be defined as in \eqref{H}. Then the following bound holds:
		\begin{equation}\label{est2}
		\|P_{\leq k}H(s-\alpha)-P_{\leq k}H(s)\|_{L^\infty}\lesssim 2^k|\alpha|\cdot\|P_{\leq k}H(s)\|_{L^\infty}
		\end{equation}
		and
		\begin{equation}\label{est2.5}
		\|P_{ k}H(s-\alpha)-P_{ k}H(s)\|_{L^2}\lesssim 2^k|\alpha|\cdot\|P_{ k}H(s)\|_{L^2}.
		\end{equation}
	\end{lemma}
	\begin{proof}
		The bound \eqref{est2} is straightforward:
  \begin{equation}
			\begin{split}
			|P_{\leq k}H(s-\alpha)-P_{\leq k}H(s)|	&\leq |\alpha|\|P_{\leq k}H'\|_{L^\infty}\lesssim 2^k|\alpha|\|P_{\leq k}H\|_{L^\infty}.
			\end{split}
\end{equation}
 Similarly, we have
		\begin{equation}
		\begin{split}
		P_{ k}H(s-\alpha)=\int_{0}^{-\alpha}P_{k}H'(s+\rho)\,d\rho+P_{k}H(s),
		\end{split}
		\end{equation}
		thus concluding the proof of the lemma.
	\end{proof}
 
	\begin{lemma}\label{lem_est3}
		Let $g\in Z_1$, $k\in\mathbb{N}$ and $L_k(s-y,\alpha)$ defined as in \eqref{Ldef}. Then we can estimate
		\begin{equation}\label{est3}
		\bigg\|\int_{\T}\frac{e^{-i\alpha/2}}{2\sin(\alpha/2)}P_kg(y)L_k(s-y,\alpha)dy-P_{k}g'(s)\frac{\alpha e^{-i\alpha}}{4\sin^2(\alpha/2)}\bigg\|_{L^2_s}\lesssim\frac{|\alpha|^2 2^k}{4|\sin^2(\alpha/2)|}\|P_kg'(s)\|_{L^2}.
		\end{equation}
	\end{lemma}
	\begin{proof}
		By writing out the kernel $L_k$ as defined in \eqref{Ln_psi}, we get
		\begin{equation}\label{est3.1}
		\begin{split}
		\int_{\T}&\frac{e^{-i\alpha}P_kg(y)L_k(s-y,\alpha)}{2\sin(\alpha/2)}dy=\frac{e^{-i\alpha}}{4\sin^2(\alpha/2)}\int_{\T}P_kg(y)\big(\psi_k(s-y)-\psi_k(s-y-\alpha) \big)dy\\
		&=\frac{e^{-i\alpha}}{4\sin^2(\alpha/2)}\int_{\T}P_{k}g(y)\int_{-\alpha}^{0}\psi_k'(s-y+\rho)d\rho dy\\
		&=\frac{e^{-i\alpha}}{4\sin^2(\alpha/2)}\int_{\T}P_{k}g(y)\int_{-\alpha}^{0}\Big(\int_{0}^\rho\psi_k''(s-y+\zeta)d\zeta +\psi_k'(s-y)\Big)d\rho dy\\
		&=\frac{e^{-i\alpha}}{4\sin^2(\alpha/2)}\int_{\T}P_{k}g(y)\int_{-\alpha}^{0}\int_{0}^\rho\psi_k''(s-y+\zeta)d\zeta d\rho dy+P_{k}g'(s)\frac{\alpha e^{-i\alpha}}{4\sin^2(\alpha/2)}\\
		&=\frac{e^{-i\alpha}}{4\sin^2(\alpha/2)}\int_{-\alpha}^{0}\int_{0}^\rho P_{k}g''(s+\zeta)d\zeta d\rho +P_{k}g'(s)\frac{\alpha e^{-i\alpha}}{4\sin^2(\alpha/2)},
		\end{split}
		\end{equation}
		where we used integration by parts in $y$ and applied the identity \eqref{identity1}.  The bounds \eqref{est3} now follow immediately.
	\end{proof}

We notice that in the lemma above, we have extracted the dependence in $\alpha$ through the term
		\begin{equation*}
		\begin{split}
		E(\alpha):=P_{k}g'(s)\frac{\alpha e^{-i\alpha}}{4\sin^2(\alpha/2)},
		\end{split}
		\end{equation*}
		for which we notice that upon integration in $\alpha$, we will get an error term due to the oddness of the integral.	When handling high-high interactions in the $g$ functions, we will need bounds for the expression
	\begin{equation}
	\mathcal{E}(s,\alpha):=\int_{\T^2}\frac{e^{-i\alpha/2}}{2\sin(\alpha/2)}P_{k_1}g_1(y_1)L_{k_1}(s-y_1,\alpha)P_{k_2}g_2(y_2)L_{k_2}(s-y_2,\alpha)\,dy_1dy_2.
	\end{equation}
	Using the modified kernel, as defined in \eqref{Ltilde}, we can decompose $\mathcal{E}$ to get
	\begin{equation}
	\begin{split}
	\mathcal{E}(s,\alpha)
	&=\int_{\T^2}\frac{e^{-i\alpha/2}}{2\sin(\alpha/2)}P_{k_1}g_1(y_1)L_{k_1}(s-y_1,\alpha)P_{k_2}g_2(y_2)\widetilde{L}_{k_2}(s-y_2,\alpha)\,dy_1dy_2\\
	&\quad+\int_{\T^2}\frac{e^{-i\alpha}}{4\sin^2(\alpha/2)}P_{k_1}g_1(y_1)L_{k_1}(s-y_1,\alpha)P_{k_2}g_2(y_2)2^{-k_2}\min\{1,2^{k_2}\alpha\}\psi_{k_2}'(s-y_2)\,dy_1dy_2\\
	&=:\mathcal{E}_1(s,\alpha)+\mathcal{E}_2(s,\alpha),
	\end{split}
	\end{equation}
 	whose bounds we collect in the following lemma.
	\begin{lemma}\label{lem_est4}
		Let $g_1,g_2\in Z_1$, $k_1,k_2\in\mathbb{N}$, and $L_{k_1}(s-y_1,\alpha)$ be defined as in \eqref{Ldef} and $\widetilde{L}_{k_2}(s-y_2,\alpha)$ as in \eqref{Ltilde}. Then
		\begin{equation}\label{est5}
		\begin{split}
		\big\|\mathcal{E}_1(s,\alpha)\big\|_{L^2_s}\lesssim\frac{2^{\min\{k_1,k_2\}/2}2^{-k_1-k_2}}{2|\sin(\alpha/2)|}\min\{2^{k_1},|\alpha|^{-1} \}\min\{2^{2k_2}|\alpha|,|\alpha|^{-1} \}\|P_{k_1}g_1'\|_{L^2}\|P_{k_2}g_2'\|_{L^2}.
		\end{split}		
		\end{equation}
		Moreover, for $2^{k_2}|\alpha|<1$ we have
  		\begin{equation}\label{est6}
		\begin{split}
		\big\|\mathcal{E}_2(s,\alpha)-P_{k_1}g_1'(s)P_{k_2}g_2'(s)\frac{\alpha^2 e^{-i\alpha}}{8\sin^3(\alpha/2)} \big\|_{L^2_s}\lesssim2^{k_2/2}\frac{2^{k_1}|\alpha|^3}{8|\sin^3(\alpha/2)|}\|P_{k_1}g_1'\|_{L^2}\|P_{k_2}g_2'\|_{L^2},
		\end{split}		
		\end{equation}
 and for $2^{k_2}|\alpha|>1$ we have
		\begin{equation}\label{est7}
		\begin{split}
		\big\|\mathcal{E}_2(s,\alpha) \big\|_{L^2_s}\lesssim\frac{2^{\min\{k_1,k_2\}/2}2^{-k_1-k_2}}{4|\sin^2(\alpha/2)|\cdot|\alpha|}\|P_{k_1}g_1'\|_{L^2}\|P_{k_2}g_2'\|_{L^2}.
		\end{split}		
		\end{equation}

	\end{lemma}
	\begin{proof}
		The bounds for \eqref{est5} follow from a straightforward application of Lemmas~\ref{lem:kernel_1}-\ref{lem:kernel2}, followed by Berstein's inequality $\|P_k f\|_{L^\infty}\lesssim2^{k/2}\|P_k f\|_{L^2}$. The bounds in \eqref{est6} are a consequence of first performing an integration by parts in $y_2$ and applying the identity \eqref{identity1}, and then applying Lemma~\ref{lem_est3}. Finally, the bounds in \eqref{est7} are a direct consequence of applying Lemma~\ref{lem:kernel_1}.
	\end{proof}

	\begin{lemma}\label{lem_est5}
		Let $g\in Z_1$ and $L_{\leq k}(s,\alpha)$ be defined as in \eqref{L_sum}. Then we have
		\begin{equation}\label{est8}
		\begin{split}
		\bigg\|\frac{d}{d\alpha}\int_{\T}P_{\leq k}g(y)L_{\leq k}(s-y,\alpha)dy\bigg\|_{L^\infty}&\lesssim \frac{1}{2|\sin(\alpha/2)|}\|P_{\leq k}g'(s)\|_{L^\infty}.
		\end{split}
		\end{equation}
	\end{lemma}
	
	\begin{proof}
		We use \eqref{Ln_psi} to write
		\begin{equation*}
		\begin{split}
		\int_{\T}P_{\leq k}g(y)L_{\leq k}(s-y,\alpha)dy&=e^{-i\alpha/2}\frac{P_{\leq k}g(s)-P_{\leq k}g(s-\alpha)}{2\sin{(\alpha/2)}},
		\end{split}
		\end{equation*}
hence taking a derivative in $\alpha$ we obtain
  		\begin{equation*}
		\begin{split}
		\frac{d}{d\alpha}\int_{\T}P_{\leq k}g(y)L_{\leq k}(s-y,\alpha)dy&=\big(P_{\leq k}g(s)-P_{\leq k}g(s-\alpha)\big)\frac{d}{d\alpha}\frac{e^{-i\alpha/2}}{2\sin{(\alpha/2)}}+e^{-i\alpha/2}\frac{P_{\leq k}g'(s-\alpha)}{2\sin{(\alpha/2)}},
		\end{split}
		\end{equation*}
which immediately implies the result.

  	\end{proof}

	\section{Analysis of the full nonlinearity}\label{sect:final}
	We consider the multi-linear paraproduct defined in \eqref{hl7} for $n+m\geq2$
	\begin{equation*}
	\begin{split}
	\mathcal{T}_k(P_{k_1}g_1,\cdots,P_{k_m}g_m,&P_{k_{m+1}}f_1,\cdots,P_{k_{n+m}}f_n)(s)\\
 &=\frac{1}{(2\pi)^m}P_k\sum_{k_1,\dots,k_m\in\mathbb{Z}}\int_{\T}\bigg(\int_{\T^m}\frac{e^{-i\alpha/2}}{2\sin(\alpha/2)} P_{k_1}g_1(y_1)\cdots P_{k_m}g_m(y_m)\\
	&\quad\times 
	L_{k_1}(s-y_1,\alpha)\cdots L_{k_m}(s-y_m,\alpha)A(\alpha)\,dy_1\cdots dy_m\\
	&\quad\times \big[P_{\leq k-4m}H(s-\alpha)+ P_{\geq k-4m}H(s-\alpha)\big]\bigg)\,d\alpha, \\
	&=\mathcal{T}_k^{L}+\mathcal{T}_k^{H},
	\end{split}
	\end{equation*}
	where $\mathcal{T}_k^{L}$ denotes the portion of the paraproduct containing the low frequency part of $H$, $P_{\leq k-4m}H$, and $\mathcal{T}_k^{H}$ denotes the remaining part.
	
	In $\mathcal{T}_k^{L}$, in order for all frequencies to add up to $k$, we must have high frequencies of order $k$ in $g$. We hence further split $$\mathcal{T}_k^{L}=\mathcal{T}_k^{L,1}+\mathcal{T}_k^{L,2}$$ where
	\begin{equation}
	\begin{split}
	\mathcal{T}_k^{L,1}(P_{k}g_1,P_{\leq k}g_2,\cdots,&P_{\leq k}g_m,P_{k_{m+1}}f_1,\cdots,P_{k_{m+n}}f_n)(s):=\frac{e^{is}}{(2\pi)^m}\int_{\T}\int_{\T^m}\frac{e^{-i\alpha/2}A(\alpha) P_{\leq k}H(s-\alpha)}{2\sin(\alpha/2)}  \\
 &\times P_{ k}g_1(y_1)L_{ k}(s-y_1,\alpha)\cdots P_{\leq k}g_m(y_m)L_{\leq k}(s-y_m,\alpha) dy_1\cdots dy_md\alpha
	\end{split}
	\end{equation}
	for when $\widetilde{g}_1$ has the highest frequency, and
	\begin{equation}
	\begin{split}
	\mathcal{T}_k^{L,2}(&P_{k_1}g_1,P_{k_2}g_2,\cdots,P_{\leq k}g_m,P_{k_{m+1}}f_1,\cdots,P_{k_{m+n}}f_n)(s):=\sum_{\substack{k_1,k_2\geq k-3 \\ |k_1-k_2|\leq6}}\frac{e^{is}}{(2\pi)^m}\int_{\T}\int_{\T^m}\frac{e^{-i\alpha/2}}{2\sin(\alpha/2)} A(\alpha) \\&\times P_{\leq k}H(s-\alpha)  P_{ k_1}g_1(y_1)L_{ k_1}(s-y_1,\alpha)\cdots P_{\leq k}g_m(y_m)L_{\leq k}(s-y_m,\alpha) dy_1\cdots dy_md\alpha,
	\end{split}
	\end{equation}
	in the case when there are high-high interactions in the $g$ functions. Here and below, we recall that
	\begin{equation}\label{L_sum}
	L_{\leq k}(s,\alpha):=\sum_{j\neq\{0,1\}}\varphi_{\leq k+2}(j)\frac{e^{-i\alpha/2}(1-e^{-i\alpha j})}{2\sin (\alpha/2)}e^{isj}.
	\end{equation}
	
	When dealing with the high frequency part of $H$, we again perform a split $$\mathcal{T}_k^{H}=\mathcal{T}_k^{H,1}+\mathcal{T}_k^{H,2}+\mathcal{T}_k^{L,1}$$ 
	to capture the possible scenarios. In the first, we have high frequencies in $g$ as well, in which case the frequency of $H$ must be of the same order as the highest frequency in the $g$ functions. We hence define
	\begin{equation}
	\begin{split}
	\mathcal{T}_k^{H,1}(P_{k_1}g_1,&P_{\leq k}g_2,\cdots,P_{\leq k}g_m,P_{k_{m+1}}f_1,\cdots,P_{k_{m+n}}f_n)(s):=\!\!\sum_{\substack{k_0,k_1\geq k-3 \\ |k_0-k_1|\leq6}}\frac{e^{is}}{(2\pi)^m}\int_{\T}\int_{\T^m}\frac{e^{-i\alpha/2}}{2\sin(\alpha/2)} A(\alpha)\\&\times P_{ k_0}H(s-\alpha)  P_{ k_1}g_1(y_1)L_{ k_1}(s-y_1,\alpha)\cdots P_{\leq k}g_m(y_m)L_{\leq k}(s-y_m,\alpha) dy_1\cdots dy_md\alpha.
	\end{split}
	\end{equation}
	In the second case, all the $g$ functions have low frequencies, thus implying that the frequency of $H$ must be of order $k$. Here we define
	\begin{equation}
	\begin{split}
	\mathcal{T}_k^{H,2}(P_{\leq k}g_1,\cdots,&P_{\leq k}g_m,P_{k_1}f_1,\cdots P_{k_n}f_{n})(s)=\frac{e^{is}}{(2\pi)^m}\int_{\T}\int_{\T^m}\frac{e^{-i\alpha/2}}{2\sin(\alpha/2)} A(\alpha) P_kH(s-\alpha)\\&\times P_{\leq k}g_1(y_1)L_{\leq k}(s-y_1,\alpha)\cdots P_{\leq k}g_m(y_m)L_{\leq k}(s-y_m,\alpha)dy_1\cdots dy_md\alpha.
	\end{split}
	\end{equation}
	Finally, since we don't have symmetry in $H$ and $g_1$, we need to add the scenario that $g_1$ is high frequency and $H$ of low frequency, which we already stated earlier.

	\subsection{The nonlinearity $\mathcal{T}_k^{L}$} 
	We begin by performing the split  
	\begin{equation}
	\mathcal{T}_k^{L}=\mathcal{T}_k^{L,\leq}+\sum_{l\in[-k,0]}\mathcal{T}_{k,l}^{L,\sim}
	\end{equation}
	where here
	\begin{equation}\label{nonlin1}
	\begin{split}
	\mathcal{T}_k^{L,\leq}(P_{k_1}g_1,\cdots,P_{\leq k}g_m,&P_{k_{m+1}}f_1,\cdots,P_{k_{m+n}}f_n)(s)=\frac{e^{is}}{(2\pi)^m}\int_{|\alpha|\leq2^{-k}}\int_{\T^m}\frac{e^{-i\alpha/2}A(\alpha) P_{\leq k}H(s-\alpha)}{2\sin(\alpha/2)}  \\
	&\times P_{ k_1}g_1(y_1)L_{ k_1}(s-y_1,\alpha)\cdots P_{\leq k}g_m(y_m)L_{\leq k}(s-y_m,\alpha) dy_1\cdots dy_md\alpha,
	\end{split}
	\end{equation}
	and
	\begin{equation}\label{nonlin2}
	\begin{split}
	\mathcal{T}_{k,l}^{L,\sim}(P_{k_1}g_1,\cdots,P_{\leq k}g_m,&P_{k_{m+1}}f_1,\cdots,P_{k_{m+n}}f_n)(s)=\frac{e^{is}}{(2\pi)^m}\int_{|\alpha|\sim2^{l}}\int_{\T^m}\frac{e^{-i\alpha/2}A(\alpha) P_{\leq k}H(s-\alpha)}{2\sin(\alpha/2)}  \\
	&\times P_{ k_1}g_1(y_1)L_{ k_1}(s-y_1,\alpha)\cdots P_{\leq k}g_m(y_m)L_{\leq k}(s-y_m,\alpha) dy_1\cdots dy_md\alpha.
	\end{split}
	\end{equation}

	\begin{lemma}\label{lem:nonlin1}
		Assume that $f_1,\cdots,f_n,g_1,\cdots,g_m\in Z_1$. Then for $k\in\mathbb{N}$ and $t\in(0,\infty)$, we have
		\begin{equation}\label{nonlin7}
		\begin{split}
		\|P_k\mathcal{T}^{L,1}\|_{L^2}&\lesssim\|P_kg_1'\|_{L^2}\prod_{i=2}^m\|g_i'\|_{L^{\infty}}\prod_{j=1}^n\|f_j'\|_{L^\infty}\\
		&\lesssim2^{-k/2}(1+2^kt)^{-2/3}\prod_{i=1}^m\|g_i\|_{Z_1}\prod_{j=1}^n\|f_i\|_{Z_1},
		\end{split}	
		\end{equation}
		and
		\begin{equation}\label{nonlin8}
		\|P_k\mathcal{T}^{L,2}\|_{L^2}\lesssim\sum_{\substack{k_1,k_2\geq k-3 \\ |k_1-k_2|\leq6}}2^{k_1/2}\|P_{k_1}g_1'\|_{L^2}\|P_{k_2}g_2'\|_{L^2}\prod_{i=3}^m\|g_i'\|_{L^{\infty}}\prod_{j=1}^n\|f_j'\|_{L^\infty}.
		\end{equation}
		Moreover, we have
		\begin{equation}\label{nonlin9}
		\|P_k\mathcal{T}^{L}(t)\|_{L^2}\lesssim2^{-k/2}(1+2^kt)^{-2/3}\prod_{i=1}^m\|g_i\|_{Z_1}\prod_{j=1}^n\|f_j\|_{Z_1}.
		\end{equation}
	\end{lemma}
	\begin{proof}
		We begin with $\mathcal{T}_k^{L,\leq,1}$ and recall that we have
		\begin{equation}
		\begin{split}
		\mathcal{T}_k^{L,\leq,1}(P_{k_1}g_1,\cdots,P_{\leq k}g_m,&P_{k_{m+1}}f_1,\cdots,P_{k_{m+n}}f_n)(s)=\frac{e^{is}}{(2\pi)^m}\int_{\T^m}\int_{|\alpha|\leq2^{-k}}\!\!\!\!\frac{e^{-i\alpha/2}A(\alpha) P_{\leq k}H(s-\alpha)}{2\sin(\alpha/2)}  \\
		&\times P_{ k}g_1(y_1)L_{ k}(s-y_1,\alpha)\prod_{j=2}^m P_{\leq k}g_j(y_j)L_{\leq k}(s-y_j,\alpha)\alpha dy_1\cdots dy_md,
		\end{split}
		\end{equation}
		We now use Lemmas~\ref{lem_est1}-\ref{lem_est3} to write
\begin{equation*}
    \begin{aligned}
        \mathcal{T}_k^{L,\leq,1}(s)&=\mathcal{T}_k^{L,\leq,1}(s)-P_{\leq k}H(s)\frac{e^{is}P_{ k}g_1'(s)}{(2\pi)^m}\prod_{j=2}^m P_{\leq k}g_j'(s)\int_{|\alpha|\leq2^{-k}} \frac{A(\alpha) \alpha e^{-i\alpha/2}}{4\sin^2{(\alpha/2)}}\prod_{j=2}^m \frac{\alpha e^{-i\alpha/2}}{2\sin(\alpha/2)}d\alpha\\
        &\quad+P_{\leq k}H(s)\frac{e^{is}P_{ k}g_1'(s)}{(2\pi)^m}\prod_{j=2}^m P_{\leq k}g_j'(s)\int_{|\alpha|\leq2^{-k}} \frac{A(\alpha) \alpha e^{-i\alpha/2}}{4\sin^2{(\alpha/2)}}\prod_{j=2}^m \frac{\alpha e^{-i\alpha/2}}{2\sin(\alpha/2)}d\alpha,
    \end{aligned}
\end{equation*}
thus using the oddness in $\alpha$ for the second line and the estimates in Lemmas~\ref{lem_est1}-\ref{lem_est3}, we obtain
		\begin{equation}\label{Tklleq1}
		\begin{split}	\|\mathcal{T}_k^{L,\leq,1}\|_{L^2_s}&\lesssim \|P_{k}g_1'\|_{L^2}\prod_{i=2}^m\|g_i'\|_{L^\infty}\|P_{\leq k}H(s)\|_{L^\infty}\int_{|\alpha|\leq 2^{-k}}\frac{|\alpha|^22^{k}}{4|\sin^2(\alpha/2)|}d\alpha\\
&\lesssim\|P_kg_1'\|_{L^2}\prod_{i=2}^m\|g_i'\|_{L^\infty}\prod_{j=1}^m\|f_j'\|_{L^\infty}.
		\end{split}
		\end{equation}
		
		We now consider $\mathcal{T}_{k,l}^{L,\sim,1}$. Using the expression of the kernel $L_{k_1}$ as given in \eqref{Ln_psi} and identity \eqref{identity1} for the highest frequency term we get
		\begin{equation}\label{nonlin3}
		\int_{\T}P_{k_1}g_1(y_1)L_{k_1}(s-y_1,\alpha)dy_1=\frac{e^{-i\alpha/2}}{2\sin(\alpha/2)}\big(P_{k_1}g_1(s)-P_{k_1}g_1(s-\alpha)\big).
		\end{equation}
	Then, from Lemma~\ref{lem:kernel_1} and Lemma~\ref{lem_est0}, we immediately get 
		\begin{equation*}
		\begin{split}
		\sum_{l\in[-k,0]}\|\mathcal{T}_{k,l}^{L,\sim,1}\|_{L^2}&\lesssim\sum_{l\in[-k,0]}\|P_{k}g_1'\|_{L^2}\|P_{\leq k}H\|_{L^\infty}\prod_{i=2}^m\|P_{\leq k}g_i'(s)\|_{L^\infty}\int_{|\alpha|\sim2^l}\frac{2^{-k}}{4|\sin^2(\alpha/2)|}d\alpha
		\\&\lesssim\sum_{l\in[-k,0]}2^{-l-k}\|P_kg_1'\|_{L^2}\prod_{i=2}^m\|g_i'\|_{L^\infty}\prod_{j=1}^n\|f_j'\|_{L^\infty}\\&\lesssim\|P_kg_1'\|_{L^2}\prod_{i=2}^m\|g_i'\|_{L^\infty}\prod_{j=1}^n\|f_j'\|_{L^\infty}
		\end{split}
		\end{equation*}
		By using the definition of the $Z_1$-norm, we obtain \eqref{nonlin7}.
		
		We now consider the scenario when we have more than one high frequency function $g$. We begin with $\mathcal{T}_{k}^{L,\leq,2}$. We get
		\begin{equation*}
		\begin{split}
		\mathcal{T}_k^{L,\leq,2}(s)&=\sum_{\substack{k_1,k_2\geq k-3 \\ |k_1-k_2|\leq6}}\frac{e^{is}}{(2\pi)^m}\int_{|\alpha|\leq2^{-k_1}}\int_{\T^m}\frac{e^{-i\alpha/2}}{2\sin(\alpha/2)} A(\alpha)\\
		&\qquad\times P_{\leq k}H(s-\alpha)  P_{ k_1}g_1(y_1)L_{ k_1}(s-y_1,\alpha)P_{ k_2}g_2(y_2)L_{ k_2}(s-y_2,\alpha)\\
		&\qquad\times\prod_{j=3}^m P_{\leq k}g_j(y_j)L_{\leq k}(s-y_j,\alpha) dy_1\cdots dy_md\alpha,
		\end{split}
		\end{equation*}
		We can now apply Lemmas~\ref{lem_est1}-\ref{lem_est2} and \ref{lem_est4} to cancel the odd part of the integral in $\alpha$ and then to estimate
		\begin{equation}\label{Tklleq2}
		\begin{aligned}
		\|\mathcal{T}_k^{L,\leq,2}\|_{L^2}&\lesssim\sum_{\substack{k_1,k_2\geq k-3 \\ |k_1-k_2|\leq6}}\|P_{\leq k}H(s)\|_{L^\infty}2^{k_1/2}\|P_{k_1}g_1'\|_{L^2}\|P_{k_2}g_2'\|_{L^2}\prod_{i=3}^m\|P_{\leq k}g_i'(s)\|_{L^\infty}\\
		&\qquad\times\int_{|\alpha|\leq 2^{-k_1}}\Big(\frac{2^k|\alpha|^3}{8|\sin^3{(\alpha/2)}|}+\frac{2^{k_2}|\alpha|}{2|\sin{(\alpha/2)}|}+\frac{2^{k_1}|\alpha|^3}{8|\sin^3{(\alpha/2)}|}\Big)d\alpha\\
		&\lesssim\sum_{\substack{k_1,k_2\geq k-3 \\ |k_1-k_2|\leq6}}2^{k_1/2}\|P_{k_1}g_1'\|_{L^2}\|P_{k_2}g_2'\|_{L^2}\prod_{i=3}^m\|g_i'\|_{L^\infty}\prod_{j=1}^n\|f_j'\|_{L^\infty}.
		\end{aligned}
		\end{equation}

		Finally, we consider $\mathcal{T}_{k,l}^{L,\sim,2}$. We recall
		\begin{equation*}
		\begin{split}
		\mathcal{T}_{k,l}^{L,\sim,2}(s)&=\sum_{\substack{k_1,k_2\geq k-3 \\ |k_1-k_2|\leq6}}\frac{e^{is}}{(2\pi)^m}\int_{|\alpha|\sim2^l}\int_{\T^m}\frac{e^{-i\alpha/2}}{2\sin(\alpha/2)} A(\alpha)\prod_{j=3}^m P_{\leq k}g_j(y_j)L_{\leq k}(s-y_j,\alpha)\\
		&\qquad\times P_{\leq k}H(s-\alpha)  P_{ k_1}g_1(y_1)L_{ k_1}(s-y_1,\alpha)P_{ k_2}g_2(y_2)L_{ k_2}(s-y_2,\alpha)dy_1\cdots dy_md\alpha.
		\end{split}
		\end{equation*}
	Substituting the expression \eqref{Ln_psi} and integrating in $y_1$ and $y_2$, followed by  Lemmas~\ref{lem:kernel_1} and \ref{lem_est0}, we obtain
		\begin{equation*}
		\begin{split}
	\|\mathcal{T}_{k,l}^{L,\sim,2}\|_{L^2}&\lesssim\sum_{\substack{k_1,k_2\geq k-3 \\ |k_1-k_2|\leq6}}
\int_{|\alpha|\sim2^l}\bigg(\|P_{\leq k}H(s)\|_{L^\infty}\frac{2^{k_2/2}2^{-k_1-k_2}}{|\sin(\alpha/2)|\cdot|\alpha|^2}\|P_{k_1}g_1'\|_{L^2}\|P_{k_2}g_2'\|_{L^2}  \\
		&\qquad\times\prod_{i=3}^m\|P_{\leq k}g_i'(s)\|_{L^\infty}\bigg)d\alpha,
		\end{split}
		\end{equation*}
  hence
  \begin{equation*}
		\begin{split}
		\sum_{l\in[-k_1,0]}\|\mathcal{T}_{k,l}^{L,\sim,2}\|_{L^2}&\lesssim\sum_{\substack{k_1,k_2\geq k-3 \\ |k_1-k_2|\leq6}}\sum_{l\in[-k_1,0]}2^{k_2/2}2^{-k_1-k_2}2^{-2l}\|P_{k_1}g_1'\|_{L^2}\|P_{k_2}g_2'\|_{L^2}\prod_{i=3}^m\|g_i'\|_{L^\infty}\prod_{j=1}^n\|f_j'\|_{L^\infty}\\
		&\lesssim\sum_{\substack{k_1,k_2\geq k-3 \\ |k_1-k_2|\leq6}}2^{k_2/2}\|P_{k_1}g_1'\|_{L^2}\|P_{k_2}g_2'\|_{L^2}\prod_{i=3}^m\|g_i'\|_{L^\infty}\prod_{j=1}^n\|f_j'\|_{L^\infty},
		\end{split}
		\end{equation*}
		thus concluding the proof of \eqref{nonlin8}.

		Finally, in order to prove \eqref{nonlin9}, we notice that by using the definition of the $Z_1$-norm \eqref{norm2}, \eqref{nonlin8} can be rewritten as 
		\begin{equation*}
		\begin{split} \|P_k\mathcal{T}^{L,2}\|_{L^2}&\lesssim\sum_{\substack{k_1,k_2\geq k-3 \\ |k_1-k_2|\leq6}}2^{k_1/2}\|P_{k_1}g_1'\|_{L^2}\|P_{k_2}g_2'\|_{L^2}\prod_{i=3}^m\|g_i'\|_{L^\infty}\prod_{j=1}^n\|f_j'\|_{L^\infty}\\
		&\lesssim\sum_{\substack{k_1,k_2\geq k-3 \\ |k_1-k_2|\leq6}}2^{k_1/2}(1+2^{k_1}t)^{-2/3}2^{-{k_1}/2}(1+2^{k_2}t)^{-2/3}2^{-k_2/2}\prod_{i=1}^m\|g_i\|_{Z_1}\prod_{j=1}^n\|f_j\|_{Z_1}\\
		&\lesssim2^{-k/2}(1+2^kt)^{-2/3}\prod_{i=1}^m\|g_i\|_{Z_1}\prod_{j=1}^n\|f_j\|_{Z_1}
		\end{split}
		\end{equation*}
		which, combined with \eqref{nonlin7} concludes the proof of the lemma.
	\end{proof}

	\subsection{The nonlinearity $\mathcal{T}_k^H$}
	We will now study the part of the nonlinearity for which the function $H$ is of high frequency. We begin with $\mathcal{T}_k^{H,1}$, the setting for which we also have high frequencies in the $g$ functions. We obtain the following result.
	\begin{lemma}\label{lemTH1}
		Assume that $f_1,\cdots,f_n,g_1,\cdots,g_m\in Z_1$. Then for $k\in\mathbb{N}$ and $t\in(0,\infty)$, we have
		\begin{equation}\label{nonlin10}
		\|P_k\mathcal{T}^{H,1}\|_{L^2}\lesssim\sum_{\substack{k_1,k_0\geq k-3 \\ |k_1-k_0|\leq6}}2^{k_0/2}\|P_{k_1}g_1'\|_{L^2}\|P_{k_0}H\|_{L^2}\prod_{i=2}^m\|g_i'\|_{L^{\infty}}.
		\end{equation}
	\end{lemma}
	\begin{proof}
		In this proof, $k_0$ denotes the frequency of the function $H$ and $k_1$ the frequency of the highest  frequency function $g_1$. 
		Similarly as before, we begin by splitting $$\mathcal{T}_k^{H,1}=\mathcal{T}_k^{H,\leq,1}+\sum_{l\in[-k_0,0]}\mathcal{T}_{k,l}^{H,\sim,1}$$ 
		with
		\begin{equation}\label{nonlin5}
		\begin{split}
		\mathcal{T}_k^{H,\leq,1}(P_{k_1}g_1,&P_{\leq k}g_2\cdots,P_{\leq k}g_m,P_{k_{m+1}}f_1,\cdots,P_{k_{m+n}}f_n)(s)\\
  &:=\sum_{\substack{k_0,k_1\geq k-3 \\ |k_0-k_1|\leq6}}\frac{e^{is}}{(2\pi)^m}\int_{|\alpha|\leq 2^{-k_0}}\int_{\T^m}\frac{e^{-i\alpha/2}}{2\sin(\alpha/2)} A(\alpha)P_{ k_0}H(s-\alpha) \\
  &\qquad\times  P_{ k_1}g_1(y_1)L_{k_1}(s-y_1,\alpha)\cdots P_{\leq k}g_m(y_m)L_{\leq k}(s-y_m,\alpha) dy_1\cdots dy_md\alpha.
		\end{split}
		\end{equation}
		and
		\begin{equation}\label{nonlin6}
		\begin{split}
		\mathcal{T}_{k,l}^{H,\sim,1}(P_{k_1}g_1,&P_{\leq k}g_2\cdots,P_{\leq k}g_m,P_{k_{m+1}}f_1,\cdots,P_{k_{m+n}}f_n)(s)\\
  &:=\sum_{\substack{k_0,k_1\geq k-3 \\ |k_0-k_1|\leq6}}\frac{e^{is}}{(2\pi)^m}\int_{|\alpha|\sim2^l}\int_{\T^m}\frac{e^{-i\alpha/2}}{2\sin(\alpha/2)} A(\alpha)P_{ k_0}H(s-\alpha)\\
  &\qquad\times   P_{ k_1}g_1(y_1)L_{ k}(s-y_1,\alpha)\cdots P_{\leq k}g_m(y_m)L_{\leq k}(s-y_m,\alpha) dy_1\cdots dy_md\alpha.
		\end{split}
		\end{equation}
		We begin by studying $\mathcal{T}_k^{H,\leq,1}$. Proceeding as for $\mathcal{T}_k^{L,\leq,1}$ and $\mathcal{T}_k^{L,\leq,2}$ in \eqref{Tklleq1} and \eqref{Tklleq2}, we apply Lemmas~\ref{lem_est1}-\ref{lem_est3} to use the oddness in $\alpha$ and estimate
		\begin{equation*}
		\begin{split}
		\|\mathcal{T}_k^{H,\leq,1}\|_{L^2}&\lesssim\sum_{\substack{k_0,k_1\geq k-3 \\ |k_0-k_1|\leq6}}\|P_{ k_0}H(s)\|_{L^\infty}\|P_{k_1}g_1'(s)\|_{L^2}\prod_{i=2}^m\|g_i'(s)\|_{L^\infty}\\
		&\qquad\times\int_{|\alpha|\leq 2^{-k_0}}\Big(\frac{2^{k_0}|\alpha|^3}{8|\sin^3(\alpha/2)|}+\frac{2^{k_1}|\alpha|^3}{8|\sin^3(\alpha/2)|}+\frac{2^{k}|\alpha|^2}{4|\sin^2(\alpha/2)|}\Big)d\alpha\\
		&\lesssim\sum_{\substack{k_0,k_1\geq k-3 \\ |k_0-k_1|\leq6}}2^{k_0/2}\|P_{ k_0}H(s)\|_{L^2}\|P_{k_1}g_1'(s)\|_{L^2}\prod_{i=2}^m\|g_i'\|_{L^{\infty}}.
		\end{split}
		\end{equation*} 
		Finally, for $\mathcal{T}_{k,l}^{H,\sim,1}$, we split the highest frequency function $g_1$ according to \eqref{nonlin3} and then use Lemmas~\ref{lem:kernel_1} and \ref{lem_est0} to obtain
		\begin{equation*}
		\begin{split}
		\sum_{l\in[-k_0,0]}\|\mathcal{T}_{k,l}^{H,\sim,1}\|_{L^2}&\lesssim\sum_{\substack{k_0,k_1\geq k-3 \\ |k_0-k_1|\leq6}}\sum_{l\in[-k_0,0]}\|P_{k_0 }H\|_{L^\infty}\|P_{k_1}g_1\|_{L^2}\prod_{i=2}^m\|P_{\leq k}g_i'\|_{L^\infty} \int_{|\alpha|\sim2^l}\frac{d\alpha}{4|\sin^2(\alpha/2)|}
		\\&\lesssim\sum_{\substack{k_0,k_1\geq k-3 \\ |k_0-k_1|\leq6}}\sum_{l\in[-k_0,0]}2^{k_0/2}2^{-l-k_1}\|P_{k_0 }H\|_{L^2}\|P_{k_1}g_1'\|_{L^2}\prod_{i=2}^m\|g_i'\|_{L^\infty}\\
		&\lesssim\sum_{\substack{k_0,k_1\geq k-3 \\ |k_0-k_1|\leq6}}2^{k_0/2}\|P_kg_1'\|_{L^2}\|P_{k_0 }H\|_{L^2}\prod_{i=2}^m\|g_i'\|_{L^\infty}
		\end{split}
		\end{equation*}
		thus concluding the proof of the lemma.
  
	\end{proof}
	The last case to consider is $\mathcal{T}_{k}^{H,2}$, the scenario in which $H$ is the only high frequency function. We obtain the following result.
	\begin{lemma}\label{lemTH2}
		Assume that $f_1,\cdots,f_n,g_1,\cdots,g_m\in Z_1$. Then for $k\in\mathbb{N}$ and $t\in(0,\infty)$, we have
		\begin{equation}\label{nonlin11}
		\begin{split}
		\|P_k\mathcal{T}^{H,2}\|_{L^2}&\lesssim\|P_{k}H\|_{L^2}\prod_{i=1}^m\|g_i'\|_{L^{\infty}}\\&\lesssim2^{-k/2}(1+2^kt)^{-2/3}\prod_{i=1}^m\|g_i\|_{Z_1}\prod_{j=1}^n\|f_i\|_{Z_1}.
		\end{split}
		\end{equation}
	\end{lemma}
	
	\begin{proof}
		We begin by splitting $\mathcal{T}_{k}^{H,2}=\mathcal{T}_{k}^{H,\leq,2}+\sum_{l\in[-k,0]}\mathcal{T}_{k,l}^{H,\sim,2}$ where here
		\begin{equation*}
		\begin{split}
		\mathcal{T}_{k}^{H,\leq,2}(P_{k_1}f_1,\cdots P_{k_n}f_{n}&,P_{\leq k}g_1,\cdots,P_{\leq k}g_m)=\frac{e^{is}}{(2\pi)^m}\int_{|\alpha|\leq 2^{-k}}\int_{\T^m}\frac{e^{-i\alpha/2}}{2\sin(\alpha/2)} A(\alpha) P_kH(s-\alpha)\\&\times P_{\leq k}g_1(y_1)L_{\leq k}(s-y_1,\alpha)\cdots P_{\leq k}g_m(y_m)L_{\leq k}(s-y_m,\alpha)dy_1\cdots dy_md\alpha,
		\end{split}
		\end{equation*}
		and
		\begin{equation*}
		\begin{split}
		\mathcal{T}_{k,l}^{H,\sim,2}(P_{k_1}f_1,\cdots P_{k_n}f_{n}&,P_{\leq k}g_1,\cdots,P_{\leq k}g_m)=\frac{e^{is}}{(2\pi)^m}\int_{|\alpha|\sim 2^{l}}\int_{\T^m}\frac{e^{-i\alpha/2}}{2\sin(\alpha/2)} A(\alpha) P_kH(s-\alpha)\\&\times P_{\leq k}g_1(y_1)L_{\leq k}(s-y_1,\alpha)\cdots P_{\leq k}g_m(y_m)L_{\leq k}(s-y_m,\alpha)dy_1\cdots dy_md\alpha.
		\end{split}
		\end{equation*}
The first term, $\mathcal{T}_k^{H,\leq,2}$ is estimated using Lemmas \ref{lem_est1} and \ref{lem_est2}. We obtain
		\begin{equation*}
		\begin{split}
		\|\mathcal{T}_{k}^{H,\leq,2}\|_{L^2}&\lesssim \|P_k H(s)\|_{L^2}\prod_{i=1}^m\|g_i'\|_{L^\infty}\int_{|\alpha|\leq 2^{-k}}\frac{2^k|\alpha|}{2|\sin{(\alpha/2)}|}d\alpha,
		\end{split}
		\end{equation*}
  where the leading term
  		\begin{equation*}
		\begin{split}
\frac{e^{is}}{(2\pi)^m}P_k H(s)\prod_{i=1}^m P_{\leq k}g_i'(s)\int_{|\alpha|\leq 2^{-k}}A(\alpha)\frac{e^{-i\alpha/2}}{2\sin{(\alpha/2)}}\Big(\frac{e^{-i\alpha/2}\alpha}{2\sin{(\alpha/2)}}\Big)^{m}d\alpha
		\end{split}
		\end{equation*}
  has been controlled once more by the oddness of the kernel.

		We now study $\mathcal{T}_{k,l}^{H,\sim,2}$. We define
		\begin{equation*}
		G'(s-\alpha):=H(s-\alpha),
		\end{equation*}  
		and rewrite
		\begin{equation*}
		\begin{split}
		\mathcal{T}_{k,l}^{H,\sim,2}(s)&=\frac{e^{is}}{(2\pi)^m}\int_{|\alpha|\sim 2^{l}}\frac{e^{-i\alpha/2}}{2\sin(\alpha/2)} A(\alpha) P_kG'(s-\alpha) \prod_{i=1}^m\int_{\T}P_{\leq k}g_i(y_i)L_{\leq k}(s-y_i,\alpha)dy_id\alpha.
		\end{split}
		\end{equation*}
		A simple change of variable enables us to perceive the derivative as a derivative in $\alpha$ and we can hence perform an integration by parts in $\alpha$ which, without loss of generality, reduces to estimating a term of the following form
		\begin{equation*}
		\begin{split}
		\widetilde{\mathcal{T}}_{k,l}^{H,\sim,2}(s)&=\frac{e^{is}}{(2\pi)^m}\int_{|\alpha|\sim 2^{l}}\bigg[\frac{d}{d\alpha}\bigg(\frac{e^{-i\alpha/2}A(\alpha)}{2\sin(\alpha/2)}\big[\varphi_0(\alpha2^{-l})+(1-\varphi_0(\alpha2^{-l}))\big] \bigg) P_kG(s-\alpha)\\
		&\qquad\times\prod_{i=1}^m\int_{\T}P_{\leq k}g_i(y_i)L_{\leq k}(s-y_i,\alpha)dy_i\bigg]d\alpha\\
		&+\frac{e^{is}}{(2\pi)^m}\int_{|\alpha|\sim 2^{l}}\bigg[\frac{e^{-i\alpha/2}}{2\sin(\alpha/2)} A(\alpha) P_kG(s-\alpha) \prod_{i=2}^m\int_{\T}P_{\leq k}g_i(y_i)L_{\leq k}(s-y_i,\alpha)dy_i\\
		&\qquad\times\frac{d}{d\alpha}\bigg(\int_{\T}P_{\leq k}g_1(y_1)L_{\leq k}(s-y_1,\alpha)dy_1\big[\varphi_0(\alpha2^{-l})+(1-\varphi_0(\alpha2^{-l}))\big]\bigg)\bigg]d\alpha.
		\end{split}
		\end{equation*}
		Here, the additional factor of $\big[\varphi_0(\alpha2^{-l})+(1-\varphi_0(\alpha2^{-l}))\big]$ is simply to ensure that we do not have any boundary terms. Using Lemmas~\ref{lem_est0} and \eqref{lem_est5}, we estimate
		\begin{equation*}
		\begin{split}
		\sum_{l\in[-k,0]}\|\widetilde{\mathcal{T}}_{k,l}^{H,\sim,2}\|_{L^2}&\lesssim\sum_{l\in[-k,0]}\|P_k G\|_{L^2}\prod_{i=1}^m\|P_{\leq k}g_i'\|_{L^\infty}\int_{|\alpha|\sim 2^{l}}\frac{d\alpha}{4|\sin^2{(\alpha/2)}|}\\
		&\lesssim\sum_{l\in[-k,0]}2^{-k-l}\|P_kH\|_{L^2}\prod_{i=1}^m\|g_i'\|_{L^\infty}\\
		&\lesssim\|P_kH\|_{L^2}\prod_{i=1}^m\|g_i'\|_{L^\infty}.
		\end{split}
		\end{equation*}
		Applying the definition of the $Z_1$-norm and Lemma~\ref{lem:kernel_1} concludes the proof of the lemma.
  
	\end{proof}
	Combining the two previous results with \eqref{nonlin7}, we get:
	\begin{lemma}\label{lem:nonlin2}
		Assume that $f_1,\cdots,f_n,g_1,\cdots,g_m\in Z_1$. Then for $k\in\mathbb{N}$ and $t\in(0,\infty)$, we have
		\begin{equation}
		\|P_k\mathcal{T}^{H}(t)\|_{L^2}\lesssim2^{-k/2}(1+2^kt)^{-2/3}\prod_{i=1}^m\|g_i\|_{Z_1}\prod_{j=1}^n\|f_j\|_{Z_1}.
		\end{equation}
	\end{lemma}
	\begin{proof}
		Using the definition of the $Z_1$-norm and Lemma~\ref{lem:kernel_1}, we see that 
		\begin{equation}
		\begin{split} \|P_k\mathcal{T}^{H,1}\|_{L^2}&\lesssim\sum_{\substack{k_0,k_1\geq k-3 \\ |k_0-k_1|\leq6}}2^{k/2}2^{k_0}\|P_{k_0}G\|_{L^2}2^{k_1}\|P_{k_1}g_1\|_{L^2}\prod_{i=2}^m\|g_i'\|_{L^\infty}\\
		&\lesssim\sum_{\substack{k_0,k_1\geq k-3 \\ |k_0-k_1|\leq6}}2^{k}(1+2^kt)^{2/3}2^{-k_0}(1+2^{k_0}t)^{-4/3}\|P_{k_0}G\|_{Z_1}\|P_{k_1}g_1\|_{Z_1}\prod_{i=3}^m\|g_i\|_{Z_1}\prod_{j=1}^n\|f_j\|_{Z_1}\\
		&\lesssim2^{-k/2}(1+2^kt)^{-2/3}\prod_{i=1}^m\|g_i\|_{Z_1}\prod_{j=1}^n\|f_j\|_{Z_1}.
		\end{split}
		\end{equation}
		Combining this with \eqref{nonlin7} and \eqref{nonlin11} concludes the proof of the lemma.
  
	\end{proof}

We are now ready to conclude the estimate of the full nonlinearity, and hence to close the fixed point argument. We denote by $e^{t\Lambda}$ the semigroup which in Fourier variables corresponds to $e^{t\mathcal{G}}$, defined in \eqref{eq1}. 
	\begin{lemma}\label{nonlin_lem}
		Assume that $f_1,\cdots,f_n,g_1,\cdots,g_m\in Z_1$. Then for $k\in\mathbb{N}$ and $t\in(0,\infty)$, we have
		\begin{equation}\label{aux1}
		\|P_k\mathcal{T}(t)\|_{L^2}\lesssim 2^{-k/2}(1+2^kt)^{-2/3}\prod_{i=1}^m\|g_i\|_{Z_1}\prod_{j=1}^n\|f_j\|_{Z_1}.
		\end{equation}
  Moreover, 
\begin{equation*}
    \begin{aligned}
        \|Y-e^{t\Lambda}Y_0\|_{Z_2}\lesssim \prod_{i=1}^m\|g_i\|_{Z_1}\prod_{j=1}^n\|f_j\|_{Z_1}.
    \end{aligned}
\end{equation*}  
	\end{lemma}

\begin{proof}
The first part \eqref{aux1} is a direct consequence of combining Lemmas~\ref{lem:nonlin1} and \ref{lem:nonlin2}.
The second part follows from the definition of the $Z_2$ norm 
 \eqref{norm3} and the nonlinear estimate \eqref{aux1}. In fact, as in \eqref{eq1}, we have
\begin{equation*}
    \begin{aligned}
        \|P_k\big(Y-e^{t\Lambda}Y_0\big)\|_{L^2}\lesssim \int_0^t e^{-C_0 2^{k}(t-\tau)}\|P_k\mathcal{T}(\tau)\|_{L^2}d\tau.
    \end{aligned}
\end{equation*}
Therefore, 
\begin{equation*}
    \begin{aligned}
        \frac{(1+2^k t)2^{3k/2}}{(2^{k}t)^{1/3}}\|P_k\big(Y-e^{t\Lambda}Y_0\big)\|_{L^2}&\lesssim \frac{(1+2^k t)2^{k}}{(2^{k}t)^{1/3}}\int_0^t \frac{e^{-C_0 2^{k}(t-\tau)}}{(1+2^k\tau)^{2/3}}d\tau \prod_{i=1}^m\|g_i\|_{Z_1}\prod_{j=1}^n\|f_j\|_{Z_1}\\
        &\lesssim  \prod_{i=1}^m\|g_i\|_{Z_1}\prod_{j=1}^n\|f_j\|_{Z_1}.
    \end{aligned}
\end{equation*}

\end{proof}

\section{Acknowledgements}

EGJ was partially supported by the RYC2021-032877 research grant, the AEI projects PID2021-125021NA-I00, PID2020-114703GB-I00 and PID2022-140494NA-I00, and the AGAUR project 2021-SGR-0087. 

SVH is partially supported by the National Science Foundation through the award DMS-2102961. Moreover, part of this material is based upon work supported by the Swedish Research Council under grant no. 2021-06594 while SVH was in residence at Institut Mittag-Leffler in Djursholm, Sweden during the  fall semester of 2023.

\bibliographystyle{amsplain2link.bst}
\bibliography{references.bib}

\providecommand{\bysame}{\leavevmode\hbox to3em{\hrulefill}\thinspace}
\providecommand{\href}[2]{#2}
\begin{thebibliography}{10}
\expandafter\ifx\csname arxiv\endcsname\relax
  \def\arxiv#1{\burlalt{arXiv:#1}{http://arxiv.org/abs/#1}}\fi
\expandafter\ifx\csname doi\endcsname\relax
  \def\doi#1{\burlalt{doi:#1}{http://dx.doi.org/#1}}\fi
\expandafter\ifx\csname href\endcsname\relax
  \def\href#1#2{#2}\fi
\expandafter\ifx\csname burlalt\endcsname\relax
  \def\burlalt#1#2{\href{#2}{#1}}\fi

\bibitem{AlazardNguyen21}
Thomas Alazard and Quoc-Hung Nguyen, \emph{On the {C}auchy problem for the
  {M}uskat equation. {II}: {C}ritical initial data}, Ann. PDE \textbf{7}
  (2021), no.~1, Paper No. 7, 25, \doi{10.1007/s40818-021-00099-x}.

\bibitem{AlazardNguyen22}
\bysame, \emph{Endpoint {S}obolev theory for the {M}uskat equation}, Comm.
  Math. Phys. \textbf{397} (2023), no.~3, 1043--1102,
  \doi{10.1007/s00220-022-04514-7}.

\bibitem{BahouriCheminDanchin11}
Hajer Bahouri, Jean-Yves Chemin, and Rapha\"{e}l Danchin, \emph{Fourier
  analysis and nonlinear partial differential equations}, Grundlehren der
  mathematischen Wissenschaften [Fundamental Principles of Mathematical
  Sciences], vol. 343, Springer, Heidelberg, 2011,
  \doi{10.1007/978-3-642-16830-7}.

\bibitem{CameronStrain23}
Stephen Cameron and Robert~M. Strain, \emph{Critical local well-posedness for
  the fully nonlinear {P}eskin problem}, Communications on Pure and Applied
  Mathematics (2023), \doi{https://doi.org/10.1002/cpa.22139}.

\bibitem{ChenNguyen23}
Ke~Chen and Quoc-Hung Nguyen, \emph{The {P}eskin problem with $\dot{
  B}^1_{\infty,\infty}$ initial data}, SIAM Journal on Mathematical Analysis
  \textbf{55} (2023), no.~6, 6262--6304, \doi{10.1137/22M1510984}.

\bibitem{GancedoGarciaJuarezPatelStrain23}
F.~Gancedo, E.~Garc\'{\i}a-Ju\'{a}rez, N.~Patel, and R.~M. Strain, \emph{Global
  regularity for gravity unstable muskat bubbles}, Mem. Amer. Math. Soc.,
  arXiv:1902.02318 (2023), To appear.

\bibitem{GancedoGraneroScrobogna21}
Francisco Gancedo, Rafael Granero-Belinch\'{o}n, and Stefano Scrobogna,
  \emph{Global existence in the {L}ipschitz class for the {N}-{P}eskin
  problem}, Indiana Univ. Math. J. \textbf{72} (2023), no.~2, 553--602,
  \doi{10.1512/iumj.2023.72.9320}.

\bibitem{GJMoriStrain23}
Eduardo Garc\'{\i}a-Ju\'{a}rez, Yoichiro Mori, and Robert~M. Strain, \emph{The
  {P}eskin problem with viscosity contrast}, Anal. PDE \textbf{16} (2023),
  no.~3, 785--838, \doi{10.2140/apde.2023.16.785}.

\bibitem{GJGSHP23}
Eduardo García-Juárez, Javier Gómez-Serrano, Susanna~V. Haziot, and Benoît
  Pausader, \emph{Desingularization of small moving corners for the {M}uskat
  equation}, Preprint arXiv:2305.05046 (2023).

\bibitem{GJKuoMoriStrain23}
Eduardo García-Juárez, Po-Chun~Kuo Kuo, Yoichiro Mori, and Robert~M. Strain,
  \emph{Well-posedness of the 3{D} {P}eskin problem}, Preprint arXiv:2301.12153
  (2023).

\bibitem{Li21}
Hui Li, \emph{Stability of the {S}tokes immersed boundary problem with bending
  and stretching energy}, J. Funct. Anal. \textbf{281} (2021), no.~9, Paper No.
  109204, 65, \doi{10.1016/j.jfa.2021.109204}.

\bibitem{LinTong19}
Fang-Hua Lin and Jiajun Tong, \emph{Solvability of the {S}tokes immersed
  boundary problem in two dimensions}, Comm. Pure Appl. Math. \textbf{72}
  (2019), no.~1, 159--226, \doi{10.1002/cpa.21764}.

\bibitem{MoriRodenbergSpirn19}
Yoichiro Mori, Analise Rodenberg, and Daniel Spirn, \emph{Well-posedness and
  global behavior of the {P}eskin problem of an immersed elastic filament in
  {S}tokes flow}, Comm. Pure Appl. Math. \textbf{72} (2019), no.~5, 887--980,
  \doi{10.1002/cpa.21802}.

\bibitem{Peskin77}
Charles~S. Peskin, \emph{Numerical analysis of blood flow in the heart}, J.
  Comput. Phys. \textbf{25} (1977), no.~3, 220--252,
  \doi{10.1016/0021-9991(77)90100-0}.

\bibitem{Peskin02}
\bysame, \emph{The immersed boundary method}, Acta Numer. \textbf{11} (2002),
  479--517, \doi{10.1017/S0962492902000077}.

\bibitem{Rodenberg18}
Analise Rodenberg, \emph{2{D} {P}eskin {P}roblems of an {I}mmersed {E}lastic
  {F}ilament in {S}tokes {F}low}, ProQuest LLC, Ann Arbor, MI, 2018, Thesis
  (Ph.D.)--University of Minnesota.

\bibitem{Tong21}
Jiajun Tong, \emph{{{R}egularized {S}tokes immersed boundary problems in two
  dimensions: Well-posedness, singular limit, and error estimates}}, Comm. Pure
  Appl. Math. \textbf{74(2):366–449} (2021).

\bibitem{Tong22}
\bysame, \emph{{Global solutions to the tangential {P}eskin problem in 2-D}},
  Preprint arXiv:2205.14723 (2022).

\bibitem{TongWei23}
Jiajun Tong and Dondyu Wei, \emph{Geometric properties of the 2-{D} {P}eskin
  problem}, Preprint arXiv:2304.09556 (2023).

\end{thebibliography}

\end{document}